\documentclass[reqno,11pt]{amsart}

\usepackage[T1]{fontenc}
\usepackage[latin1]{inputenc}
\usepackage{amsmath}
\usepackage{amssymb}
\usepackage{amsthm}
\usepackage{enumerate}
\usepackage{tikz}
\usetikzlibrary{patterns,intersections}
\usepackage{color}
\usepackage[colorlinks,urlcolor=blue]{hyperref}

\definecolor{gray3}{gray}{0.60}

\newcommand{\Cstar}[3]{\ensuremath{C^{\ast}([#1]^{#2} \times [2]^{\ell}, #3)}}
\newcommand{\dCstar}[3]{\ensuremath{C^{\ast}\bigl([#1]^{#2} \times [2]^{\ell}, #3\bigr)}}
\newcommand{\onetotheell}{\mathbf{1}^{\ell}}
\newcommand{\one}[1]{\mathbf{1}^{#1}}
\newcommand{\growth}{\mathcal{G}}
\newcommand{\bvec}{\mathbf{b}}
\newcommand{\seq}{\mathcal{E}}

\newcommand{\constant}[3]{\ensuremath{c_{#1, #2, #3}}}

\newcommand{\E}{\mathbb{E}}
\newcommand{\N}{\mathbb{N}}
\newcommand{\Prob}{\mathbb{P}}
\newcommand{\R}{\mathbb{R}}
\newcommand{\Z}{\mathbb{Z}}

\newcommand{\DD}{\mathcal{D}}
\newcommand{\TT}{\mathcal{T}}

\DeclareMathOperator{\Bin}{Bin}

\numberwithin{equation}{section}

\newtheorem*{Harris}{Harris's Lemma}
\newtheorem{theorem}{Theorem}[section]
\newtheorem{proposition}[theorem]{Proposition}
\newtheorem{lemma}[theorem]{Lemma}
\newtheorem{claim}[theorem]{Claim}
\newtheorem{conjecture}[theorem]{Conjecture}

\theoremstyle{remark}
\newtheorem{remark}[theorem]{Remark}

\title[Upper Bound for Bootstrap Percolation in All Dimensions]{An Improved Upper Bound for Bootstrap Percolation in All Dimensions}

\author{Andrew J. Uzzell}
\address{Department of Mathematics and Statistics\\Grinnell College\\Grinnell, IA 50112, USA}
\email{\href{mailto:uzzellan@grinnell.edu}{uzzellan@grinnell.edu}}
\thanks{This work was done while the author was at the University of Memphis.}

\begin{document}

\begin{abstract}
In $r$-neighbor bootstrap percolation on the vertex set of a graph~$G$, a set~$A$ of initially infected vertices spreads by infecting, at each time step, all uninfected vertices with at least~$r$ previously infected neighbors.  When the elements of~$A$ are chosen independently with some probability~$p$, it is natural to study the critical probability $p_c(G,r)$ at which it becomes likely that all of~$V(G)$ will eventually become infected.  Improving a result of Balogh, Bollob\'as, and Morris, we give a bound on the second term in the expansion of the critical probability when $G = [n]^d$ and $d \geq r \geq 2$.  We show that for all~$d \geq r \geq 2$ there exists a constant~$c_{d,r} > 0$ such that if $n$ is sufficiently large, then
\[
p_c([n]^d, r) \leq \Biggl(\dfrac{\lambda(d,r)}{\log_{(r-1)}(n)} - \dfrac{c_{d,r}}{\bigl(\log_{(r-1)}(n)\bigr)^{3/2}}\Biggr)^{d-r+1},
\]
where $\lambda(d,r)$ is an exact constant and $\log_{(k)}(n)$ denotes the $k$-times iterated natural logarithm of~$n$.
\end{abstract}

\maketitle

\section{Introduction}\label{se:intro}

Bootstrap percolation on the vertex set of a graph is a cellular automaton in which vertices have two possible states, ``infected'' and ``uninfected''.  Let $r \in \N$ and let $G$ be a graph.  In \emph{$r$-neighbor bootstrap percolation}, a set~$A \subseteq V(G)$ is infected at time~$0$.  At each subsequent time step, all infected vertices remain infected and all uninfected vertices with at least~$r$ infected neighbors become infected.  In symbols, letting $A_0 = A$, we have
\[
A_{t+1} = A_t \cup \bigl\{v \, : \, \lvert N(v) \cap A_t \rvert \geq r\bigr\}
\]
for all~$t \geq 0$.  Define the \emph{closure} of~$A$ to be $[A] := \bigcup_{t=0}^{\infty} A_t$, the set of vertices that eventually become infected.  If $[A] = V(G)$, we say that $A$ \emph{percolates}~$G$, or simply that $G$ \emph{percolates}.

Bootstrap percolation was introduced by Chalupa, Leath, and Reich~\cite{ChaLeathReich} in connection with the Blume--Capel model of ferromagnetism.

Here, as often in the literature, elements of~$A$ are chosen independently with some probability~$p$.  Given $p \in [0,1]$, we define $P(G, r, p)$ to be the probability that $A$ percolates $G$ under the $r$-neighbor rule if the elements of~$A$ are chosen in this way.  We define $p_{\alpha}(G, r) = \inf \{p : P(G, r, p) \geq \alpha\}$ and let
\[
p_c := p_{c}(G, r) := p_{1/2}(G, r)
\]
denote the \emph{critical probability}.

Van Enter~\cite{vanEnter} and Schonmann~\cite{Schon} showed that for all~$d \geq 2$, $p_c(\Z^d, r) = 0$ when $r \leq d$ and $p_c(\Z^d, r) = 1$ when $r \geq d + 1$.  Since then, much work has focused on $p_c([n]^d, r)$ (where $[n] = \{1, \ldots, n\}$) for $2 \leq r \leq d$.  In this case, the critical probability of percolation displays a sharp threshold.
That is, for all~$\varepsilon > 0$ and $n$~sufficiently large, if $p > (1 + \varepsilon)p_c$, then $P([n]^d, r, p) > 1- \varepsilon$,
and if $p < (1 - \varepsilon)p_c$, then $P([n]^d, r, p) < \varepsilon$.
Aizenman and Lebowitz~\cite{AizLeb} determined that for all~$d \geq 2$, $p_c([n]^d, 2) = \Theta(1/\log n)^{d-1}$.  (Later, Balogh and Pete~\cite{BalPete} independently proved this result for $d = 2$.)  Cerf and Cirillo~\cite{CerfCir} and Cerf and Manzo~\cite{CerfManzo} showed that for all~$d \geq r \geq 2$, $p_c([n]^d, r) = \Theta\bigl(1/\log_{(r-1)}(n)\bigr)^{d-r+1}$, where $\log_{(k)}(n)$ denotes the $k$-times iterated natural logarithm of~$n$, so that $\log_{(k)}(n) = \log (\log_{(k-1)}(n))$ and $\log_{(1)}(n) = \log n$.

The next breakthrough in the field was due to Holroyd~\cite{Hol}, who proved a sharp threshold result for bootstrap percolation on the two-dimensional grid.

\begin{theorem}\label{thm:Holroyd}
As $n \rightarrow \infty$,
\[
p_c\bigl([n]^2, 2\bigr) = \dfrac{\pi^2}{18\log n} + o\biggl(\dfrac{1}{\log n}\biggr).
\]
\end{theorem}

Later, Gravner and Holroyd~\cite{GravHol}, Gravner, Holroyd, and Morris~\cite{GravHolMor}, and Hartarsky and Morris~\cite{HartarskyMorris} sharpened Holroyd's result.  Collectively, they proved the following.

\begin{theorem}\label{thm:2Dsharper}
There exist constants~$C \geq c > 0$ such that
\[
\dfrac{\pi^2}{18\log n} - \dfrac{C}{(\log n)^{3/2}} \leq p_c\bigl([n]^2, 2\bigr) \leq \dfrac{\pi^2}{18\log n} - \dfrac{c}{(\log n)^{3/2}}
\]
for all $n$~sufficiently large.
\end{theorem}

Turning to higher dimensions, Balogh, Bollob\'as, and Morris~\cite{BBM3D} proved a sharp threshold result for $p_c([n]^3, 3)$ and proved an upper bound on $p_c([n]^d, r)$ for all constant $d \geq r \geq 2$.  Later, Balogh, Bollob\'as, Duminil-Copin, and Morris~\cite{BBDCM} proved the corresponding lower bound and so established the sharp threshold result for all constant $d \geq r \geq 2$.  These results are substantially more difficult than the proof of Theorem~\ref{thm:Holroyd}.

Before we state the results of~\cite{BBDCM,BBM3D}, we need to introduce more notation.  Given $k \geq 1$, define the function~$\beta_k : (0, 1) \to (0, 1)$ by
\begin{equation}\label{eq:betadef}
\beta_k(u) = \frac{1}{2} \Bigl(1 - (1 - u)^k + \sqrt{1 + (4u - 2)(1 - u)^k + (1 - u)^{2k}}\Bigr)
\end{equation}
and let
\begin{equation}\label{eq:gdef}
g_k(z) = -\log\bigl(\beta_k\bigl(1 - e^{-z}\bigr)\bigr).
\end{equation}
For $d \geq r \geq 2$, define
\begin{equation}\label{eq:lambdaDef}
\lambda(d,r) = \int_0^{\infty} g_{r-1}\bigl(z^{d-r+1}\bigr)\,dz.
\end{equation}
Holroyd~\cite{Hol} showed that $\lambda(2,2) = \pi^2/18$.  At present, $(2, 2)$ is the only ordered pair~$(d, r)$ for which an exact expression for $\lambda(d, r)$ is known.  However, it is shown in~\cite{BBM3D} that $\lambda(d, r) < \infty$ for all~$d \geq r \geq 2$.

Here is the sharp threshold result of Balogh, Bollob\'as, Duminil-Copin, and Morris~\cite{BBDCM,BBM3D}.

\begin{theorem}\label{thm:AllDthreshold}
Let $d \geq r \geq 2$.  With $\lambda(d,r)$ as defined in~\eqref{eq:lambdaDef},
\[
p_c\bigl([n]^d, r\bigr) = \biggl(\dfrac{\lambda(d,r) + o(1)}{\log_{(r-1)}(n)}\biggr)^{d-r+1}.
\]
\end{theorem}


A number of variations of the bootstrap process described above have been considered.  Holroyd~\cite{Hol,HolModAllD} proved, for all~$d \geq 2$, a sharp threshold result for a modified $d$-neighbor bootstrap rule on $[n]^d$: in order to become infected, a vertex must have at least~one infected neighbor in each dimension.
Sharp threshold results have also been proved for other update rules on $\Z^d$ and $[n]^d$~\cite{BDCMS:Duarte,DCvE,DCvEH,DCHol,vEF}.  Similar but weaker results about the threshold behavior of a very general class of update rules on $\Z^2$ were proved in~\cite{BBPS,BDCMS,BSU}.

Bootstrap percolation has been applied to other fields, especially physics.  In particular, there is a strong connection between bootstrap percolation and the Glauber dynamics of the Ising model of ferromagnetism at zero temperature~\cite{AizLeb,FSSIsing,MorrisGlauber}.  For other applications in physics, see~\cite{AdLev} and the references therein.

In~\cite{BBDCM}, Balogh, Bollob\'as, Duminil-Copin, and Morris suggested that the techniques of~\cite{GravHolMor} could be used to prove an analogue of Theorem~\ref{thm:2Dsharper} for $p_c([n]^d, 2)$.  We carry this program out in part.  We combine the techniques of~\cite{GravHol} and~\cite{BBM3D} to improve the upper bound on $p_c([n]^d, r)$ given in Theorem~\ref{thm:AllDthreshold} for all~$d \geq r \geq 2$.

\begin{theorem}\label{thm:UB}
For all~$d \geq r \geq 2$, there exists a constant~$c_{d,r} > 0$ such that 
\begin{equation}\label{eq:UB}
p_c([n]^d, r) \leq \Biggl(\dfrac{\lambda(d,r)}{\log_{(r-1)}(n)} - \dfrac{c_{d,r}}{\bigl(\log_{(r-1)}(n)\bigr)^{3/2}}\Biggr)^{d-r+1}
\end{equation}
for all $n$~sufficiently large.
\end{theorem}

We note that when $d = r = 2$, \eqref{eq:UB} reduces to the upper bound in Theorem~\ref{thm:2Dsharper}, which was proved by Gravner and Holroyd~\cite{GravHol}.

The rest of the paper is organized as follows.  In Section~\ref{se:outline}, we give an outline of the proof of Theorem~\ref{thm:UB}.  In Section~\ref{se:notation}, we introduce additional notation and some preliminary results.  In Section~\ref{se:auxiliary}, we state an important auxiliary result, Theorem~\ref{thm:cstar}, and also state and prove other auxiliary results.  In Section~\ref{se:proofbasecase}, we prove Theorem~\ref{thm:cstar} in the case~$r = 2$.  In Section~\ref{se:proofs}, we complete the proof of Theorem~\ref{thm:cstar} and use it to deduce Theorem~\ref{thm:UB}.  Finally, in Section~\ref{se:open}, we conjecture an improved lower bound on $p_c([n]^d, r)$.

\section{Outline of the Proof of Theorem~\ref{thm:UB}}\label{se:outline}

Here we will sketch the proof of Theorem~\ref{thm:UB}.  Our argument builds on a large body of previous work (in particular, \cite{AizLeb}, \cite{Hol}, \cite{GravHol}, \cite{CerfManzo}, and~\cite{BBM3D}).  We hope that discussing the relevant ideas from these papers at some length will serve to make our proof clearer to the reader.

We begin with a few definitions.  In the literature of percolation theory, vertices of a graph are often called \emph{sites}, and we will almost always use this term hereafter.
We say that a set~$S \subseteq [n]^d$ is \emph{internally spanned} if $[A  \cap S] = S$.  We say that a set of vertices is \emph{empty} or \emph{unoccupied} if it contains no infected sites and \emph{occupied} otherwise.  We say that a sequence of events $E_1$, \dots,~$E_n$ has a \emph{double gap} if some pair of consecutive events~$(E_i, E_{i+1})$ does not occur.  Finally, given $r \geq 3$, let $\one{r-2}$ denote the vector~$(1, \ldots, 1) \in \R^{r - 2}$ and, for each~$i \in [r - 2]$, let $e_i$ denote the $i$th standard basis vector.  

\subsection{Two Dimensions}\label{subsec:2D}

One might suppose that if $[n]^d$ percolates, then the infected set spreads to all parts of the grid in a fairly uniform manner.  In~\cite{AizLeb}, Aizenman and Lebowitz showed that in fact, when the infection probability~$p$ is on the order of~$(1/\log n)^{d-1}$, whether percolation occurs under the $2$-neighbor rule is governed by a more local phenomenon: the existence of a fairly small internally spanned set, called a \emph{critical droplet}.
For example, for $2$-neighbor percolation in $[n]^2$, a natural candidate for a critical droplet is a rectangle whose diameter (in the $L^{\infty}$ norm) is on the order of~$\log n$.  (A heuristic explanation for this is given in Section~\ref{subsec:moreD}.)

So, in~\cite{Hol}, Holroyd proved the upper bound on $p_c([n]^2, 2)$ in Theorem~\ref{thm:Holroyd} by estimating the probability that a square~$R$ of side length~$B :\approx \log n$ is internally spanned in a certain way.

Let $a \ll B$ and let $S$ denote the copy of~$[a]^2$ in the lower left corner of~$R$.  If $S$ is fully infected, what conditions imply that the infected set will grow from $S$ to fill $R$?  If the rows $[a] \times \{a + 1\}$, $[a] \times \{a + 2\}$, \dots,~$[a] \times \{B\}$ are all occupied, then these rows will iteratively become infected.  If the same holds for the columns~$\{i\} \times [a]$, then all of $[B]^2$ will become infected.

Holroyd observed that we can get away with asking for a bit less.  Note that if either of the rows $[a] \times \{a + 1\}$~and~$[a] \times \{a + 2\}$ contains an infected site, then all sites in the row~$[a] \times \{a + 1\}$ will become fully infected.  (Much the same is true for the columns $\{a + 1\} \times [a]$~and~$\{a + 2\} \times [a]$.)  Motivated by this observation, we let $R_i$ denote the event that $[i-1] \times \{i\}$ is occupied, let $C_i$ denote the event that $\{i\} \times [i-1]$ is occupied, and let $D$ denote the event that the sequences $(R_i)_{i = a + 1}^{B + 1}$ and~$(C_i)_{i = a + 1}^{B + 1}$ each contain no double gaps.  Observe that if $D$ occurs, then the infected set will grow from $S$ to fill $R$.  We think of $D$ as ``diagonal growth'' of the infected set, because the infected set iteratively fills the sets~$[t]^2$ for $t = a + 1$, \dots,~$B$ (see Figure~\ref{fig:diag}).

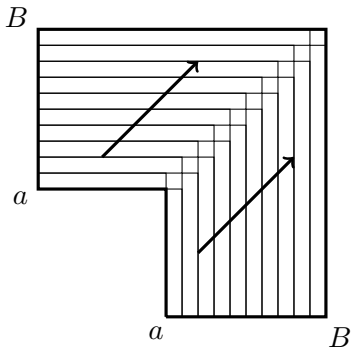
\begin{figure}
\begin{tikzpicture}[scale=0.85]
	\path[shape=coordinate]
		(2.125,0) coordinate (a1)
		(4.375,0) coordinate (b1)
		(0,2.125) coordinate (1a)
		(0,4.375) coordinate (1b);
	\node at (b1) [below right] {$B$};
	\node at (1b) [above left] {$B$};
	\node at (a1) [below left] {$a$};
	\node at (1a) [below left] {$a$};
	\foreach \a in {2,2.25,...,4.25}
		{
			\draw (\a,0) rectangle (\a + 0.25,\a);
		}
	\foreach \a in {2,2.25,...,4.25}
		{
			\draw (0,\a) rectangle (\a,\a + 0.25);
		}
	\draw[very thick] (2,0) -- (4.5,0) -- (4.5,4.5) -- (0,4.5)  -- (0,2) -- (2,2) -- (2,0);
	\draw[->,very thick] (2.5,1) -- ++(1.5,1.5);
	\draw[->,very thick] (1,2.5) -- ++(1.5,1.5);
\end{tikzpicture}
\caption[Diagonal growth of the infected set.]{If no two consecutive rows and no two consecutive columns are unoccupied, then the infected set will grow diagonally.}\label{fig:diag}
\end{figure}

As the reader might guess, the probability that $D$ occurs is fairly small.  However, it is large enough that if $[n]^2$ is partitioned into squares of side length~$B$, then with high probability $D$ occurs in \emph{some} square.  Furthermore, if such a square is fully infected, then with high probability the infected set will fill the entire grid.

How might one prove a stronger upper bound on $p_c([n]^2, 2)$?  Instead of considering a single event that implies that the square~$R$ is internally spanned, one might consider a set of pairwise disjoint events $E_1$, \dots, $E_N$, for some $N = N(p)$, each of which implies that $R$ is internally spanned.  If, for each $i$, we had $\Prob(E_i) \geq (c_1 p)^{1/p^2}\Prob(D)$, and if $N = (c_2/p)^{1/p^2}$ (where $c_1$ and $c_2$ are constants such that $c_1 c_2 > 1$), then we would have
\begin{equation}\label{eq:Heuristic2D}
\Prob\biggl(\bigvee_{i=1}^N E_i\biggr) \geq (c_2/p)^{1/p^2}(c_1 p)^{1/p^2}\Prob(D) = e^{c/p^2}\Prob(D),
\end{equation}
where $c := \log(c_1 c_2) > 0$.  It turns out that the factor~$e^{c/p^2}$ on the right-hand side of~\eqref{eq:Heuristic2D} is enough to make a difference in the value of~$p_c([n]^2, 2)$ and to prove the upper bound in Theorem~\ref{thm:2Dsharper}.

Gravner and Holroyd~\cite{GravHol} did precisely this.  They considered the event that $R$ is internally spanned, but that at some point, a double gap in either $(R_i)$ or $(C_i)$ creates a small ``detour'' in the diagonal growth of the infected set.  Once again, consider the fully infected square~$S = [a]^2$, and suppose that for some $b > a$, the rows $[b - 1] \times \{a + 1\}$ and~$[b - 1] \times \{a + 2\}$ are both empty.  Clearly, this double gap blocks the infected set from growing vertically.  However, if the columns to the right of~$S$ contain no double gaps until at least column~$b + 1$, then the infected set can grow horizontally until it fills the rectangle~$[b] \times [a]$.  If the infected set eventually encounters an infected site above $[b] \times [a]$ (for example, $(b, a + 2)$), then it can overcome the double gap and fill the rows $[b] \times \{a + 1\}$ and~$[b] \times \{a + 2\}$.  Finally, if there are no further double gaps in the rows above $[b] \times [a + 2]$, then the infected set can grow vertically until it fills $[b]^2$ (see Figure~\ref{fig:alternative}).

\begin{figure}
\centering
\begin{tikzpicture}[scale=0.75]
	\path[shape=coordinate]
		(1.875,0) coordinate (a1)
		(4.375,0) coordinate (b1)
		(0,1.875) coordinate (1a)
		(0,4.375) coordinate (1b);
	\node at (b1) [below right] {$b$};
	\node at (1b) [above left] {$b$};
	\node at (a1) [below left] {$a$};
	\node at (1a) [below left] {$a$};
	\draw[->,very thick] (2.25,2.5) -- (2.25,4.25);
	\draw[->,very thick] (2,0.875) -- (4.25,0.875);
	\draw[fill=gray!20] (0,1.75) rectangle (4.5,2.25);
	\filldraw[fill=black!70] (4.25,2) rectangle (4.5,2.25);
	\draw[very thick] (1.75,0) -- (4.5,0) -- (4.5,4.5) -- (0,4.5)  -- (0,1.75) -- (1.75,1.75) -- (1.75,0);
\end{tikzpicture}
\caption{An alternative way of filling a rectangle.
The light gray region is unoccupied and the dark gray square represents a single infected site.  The arrows depict the growth of the infected set across regions with no double gaps.}\label{fig:alternative}
\end{figure}
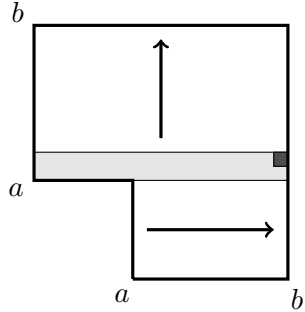

It is not hard to show that such a ``detour'' is less probable than the event~$D$ defined above.  However, Gravner and Holroyd showed that if $a$ and~$b$ are both on the order of~$1/p$ and $b - a =O(1/\sqrt{p})$, then these detours (or, more precisely, sequences of such detours) are both probable and numerous enough that~\eqref{eq:Heuristic2D} holds.

\subsection{Higher Dimensions}\label{subsec:moreD}

Now we will describe the proof of the upper bound on $p_c([n]^d, r)$ given in~\cite{BBM3D} and discuss how we will adapt it to prove Theorem~\ref{thm:UB}.

In order to prove the upper bound in Theorem~\ref{thm:AllDthreshold}, Balogh, Bollob\'as, and Morris~\cite{BBM3D} also used the notion of a critical droplet.  They observed that it follows from the results in~\cite{CerfManzo} that a critical droplet for $r$-neighbor percolation in $[n]^d$ is a $d$-dimensional box whose diameter is about~$\log n$.
As a heuristic justification for this,
let $X_L$ denote the number of internally spanned cubes of diameter~$L$ in $[n]^d$.
It is shown in~\cite{CerfManzo} that if $L$ is in a certain range, then the probability that a cube of diameter~$L$ is internally spanned is (\emph{very} roughly)~$e^{-L}$.  Thus, if $L \approx d\log n$, then
$\E X_L \approx n^{d+o(1)} e^{-L} = \Theta(1)$, which suggests that
the ``critical diameter'' is indeed on the order of~$\log n$.

Suppose, then, that a cube~$T_0 \cong [\log n]^d$ is internally spanned.
Under what circumstances is it likely that the infected set can grow from $T_0$?   In particular, when will the infected set grow to fill the $(d - 1)$-dimensional ``layer'' that is adjacent to~$T_0$ in a given direction?  Choose a direction and let $S_1$ denote the two layers adjacent to~$T_0$ in this direction.  Observe that $S_1 \cong [\log n]^{d-1} \times [2]$.  Crucially, because each site in the layer of~$S_1$ adjacent to~$T_0$, $[\log n]^{d-1} \times \{1\}$, already has an infected neighbor in $T_0$, each such site requires only $r - 1$ additional infected neighbors in $S_1$ in order to become infected.  In contrast, sites in the other layer of $S_1$, $[\log n]^{d-1} \times \{2\}$, still require $r$ infected neighbors.

Therefore, it makes sense to consider percolation inside $S_1$, where each site in the first layer has infection threshold~$r - 1$ and each site in the second layer has threshold~$r$.  When is it likely that the first layer of~$S_1$ will percolate?  By applying the same heuristic argument as above to $S_1$, we see that a good candidate for a critical droplet in $S_1$ is a set of the form~$[\log\log n]^{d-1} \times [2]$.  (Here, the term ``critical droplet'' has a slightly different meaning: it refers to a set whose first layer, if fully infected, will with high probability infect the rest of the first layer of~$S_1$.)

Suppose, then, that $S_1$ contains a set~$T_1 \cong [\log\log n]^{d-1} \times [2]$ whose first layer is fully infected.  What is required for the infected sites in $T_1$ to fully infect the first layer of~$S_1$?  As before, percolation must occur in the first layer of each copy~$S_2$ of~$[\log\log n]^{d-2} \times [2]^2$ that is adjacent to~$T_1$ (and contained in $S_1$).  Note, however, that sites in the first layer of~$S_2$, $[\log\log n]^{d-2} \times \{(1, 1)\}$, need only $r - 2$ infected neighbors in $S_2$, because each such site has one infected neighbor in $T_0$ and another in $T_1$.  In contrast, sites in
the other layers of~$S_2$
still require $r$ infected neighbors in $S_2$ in order to become infected.\footnote{Sites in $[\log\log n]^{d-2} \times \{(2, 1)\}$ actually only need $r - 1$ infected neighbors in $S_2$, because each one has an infected neighbor in $T_0$, but it turns out that we lose nothing by assuming that these sites also have infection threshold~$r$.}

Iterating this argument leads us to consider the probability of percolation in a set of the form $[\log_{(r-2)}(n)]^{d - r + 2} \times [2]^{r - 2}$, where all sites
in $[\log_{(r-2)}(n)]^{d - r + 2} \times \{(1, \dots, 1)\}$
have threshold~$2$ and all other sites have threshold~$r$.  By induction, if it is likely that all sites in $[\log_{(r-2)}(n)]^{d - r + 2} \times \{(1, \dots, 1)\}$ become infected, then percolation is likely to occur in $[n]^d$.

Balogh, Bollob\'as, and Morris bounded the probability that percolation occurs in $[\log_{(r-2)}(n)]^{d - r + 2} \times [2]^{r - 2}$ in much the same way that Holroyd bounded the probability of $2$-neighbor percolation in $[n]^2$.  Let $B \gg a$ and suppose that a cube~$K \cong [a]^{d-r+2} \times \one{r-2}$ is fully infected.  In order to estimate the probability that $[B]^{d-r+2} \times \one{r-2}$ becomes infected, we would like a fairly simple sufficient condition for all of the sites in a layer adjacent to~$K$ (for example, $[a]^{d-r+1} \times \{a + 1\} \times \one{r-2}$) to become infected.  Let $U_i$ denote the event that $[a]^{d-r+1} \times \{i\} \times \one{r-2}$ is occupied and, for each $j \in [r - 2]$, let $V_i^{(j)}$ denote the event that $[a]^{d-r+1} \times \{i\} \times (\one{r-2} + e_j)$ is occupied.  Observe that because each site in $[a]^{d-r+1} \times \{a + 1\} \times \one{r-2}$ has threshold~2 and already has an infected neighbor in $[a]^{d-r+2} \times \one{r-2}$, all of the sites in this layer will become infected if one of the events $U_{a+1}$, $U_{a+2}$, $V_{a+1}^{(1)}$, \ldots,~$V_{a+1}^{(r-2)}$ occurs.  We call the situation in which none of these events occur---or, more generally, the event
\[
\neg \bigl(U_{i} \vee U_{i+1} \vee V_{i}^{(1)} \vee \cdots \vee V_{i}^{(r-2)}\bigr)
\]
for any $i \in [B]$---an \emph{$L$-gap}.\footnote{When $d = r = 3$, the sets in question are $[a] \times \{(i, 1)\}$, $[a] \times \{(i + 1, 1)\}$, and~$[a] \times \{(i, 2)\}$.  An $L$-gap is so called because these sets form an $L$-shape when viewed from the side.}
(Note that when $d = r = 2$, an $L$-gap is simply a double gap.)  If the sequence
\[
\seq_{d-r+2} := (U_i)_{a+1 \leq i \leq B+1} \cup (V_i^{(j)})_{a+1 \leq i \leq B, j \in [r-2]}
\]
contains no $L$-gaps, then the infected set will grow in direction~$d - r + 2$ until it reaches one face of~$[B]^{d-r+2} \times \one{r-2}$.  We may similarly define a sequence~$\seq_t$ for each direction~$t \in [d - r + 2]$.  If none of the $\seq_t$ contains an $L$-gap, then the infected set will fill $[B]^{d-r+2} \times \one{r-2}$.

To bound the probability of percolation from below, it suffices to show that if $B$ is (roughly) on the order of~$\log_{(r-1)}(n)$, then the probability that $[B]^{d-r+2} \times \one{r-2}$ is internally spanned in the manner described above is large enough that with high probability, some cube in $[\log_{(r-2)}(n)]^{d - r + 2} \times \one{r-2}$ of side length~$B$ is internally spanned.  If so, then with high probability, the rest of $[\log_{(r-2)}(n)]^{d - r + 2} \times \one{r-2}$ will become infected.  This proves the upper bound in Theorem~\ref{thm:AllDthreshold} for~$r = 2$.  The upper bound for larger~$r$ follows from an inductive argument that shows that the full infection of the first layer of~$[\log_{(r-2)}(n)]^{d - r + 2} \times [2]^{r - 2}$ indeed implies $r$-neighbor percolation in $[n]^d$.

As mentioned above, in order to prove Theorem~\ref{thm:UB}, we unite the techniques of \cite{BBM3D}~and~\cite{GravHol}.  Again, we consider a fully infected cube~$K \cong [a]^{d-r+2} \times \one{r-2}$.  Let $D'$ denote the event that the infected set grows from $K$ without encountering $L$-gaps until it fills a cube~$[B]^{d-r+2} \times \one{r-2}$.
Just as in the two-dimensional case, we seek a large class of pairwise disjoint events $E'_1$, \dots, $E'_{N'}$, each of which implies that $[B]^{d-r+2} \times \one{r-2}$ becomes fully infected, such that for all~$i$, we have $\Prob(E'_i) \geq (c_1 p)^{1/p^{2(d-r+1)}} \Prob(D')$.  If, moreover, we can show that $N' = N'(p) = (c_2/p)^{1/p^{2(d-r+1)}}$ for an appropriate constant~$c_2$, then, similarly to~\eqref{eq:Heuristic2D}, we will be able to conclude that
\begin{equation}\label{eq:HeuristicHigherD}
\Prob\biggl(\bigvee_{i=1}^{N'} E'_i\biggr) \geq e^{c/p^{2(d-r+1)}}\Prob(D')
\end{equation}
for some constant~$c > 0$.

Much as in~\cite{GravHol}, we consider events~$E'_i$ that involve small detours in the growth of the infected set.  Suppose that an $L$-gap---for example, $\neg(U_{a+1} \vee U_{a+2} \vee V_{a+1}^{(1)} \vee \cdots \vee V_{a+1}^{(r-2)})$---blocks the infected set from growing from the cube~$K$ in direction~$d - r + 2$.  If no $L$-gaps occur in the other sequences~$\seq_t$, then the infected set may be able to grow in the other $d - r + 1$ directions until it fills a set of the form~$[b]^{d-r+1} \times [a] \times \one{r-2}$, for some $b > a$.  If there is an infected site~$x$ with $x_{d-r+1} \in \{a + 1, a + 2\}$ (for example, $(b, \ldots, b, a + 2) \times \one{r - 2}$), then the infected set can overcome the $L$-gap and fill $[b]^{d-r+1} \times [a + 2] \times \one{r-2}$.  If no further $L$-gaps occur in direction~$d - r + 2$, then the infected set will grow in that direction until it fills the cube~$[b]^{d-r+2} \times \one{r-2}$.

We show that when $a$ and~$b$ are both on the order of~$p^{-1/(d - r + 1)}$ and $b - a = O(p^{-1/2(d - r + 1)})$, then the number and probability of these detours (or, rather, of sequences of such detours) are both large enough that $\Prob(\bigvee_{i=1}^{N'} E'_i)$ satisfies~\eqref{eq:HeuristicHigherD}.  This yields the claimed improvement in the upper bound on $p_c([n]^d, 2)$.



The rest of the proof of Theorem~\ref{thm:UB} consists of an inductive argument that is very similar to the inductive argument of~\cite{BBM3D} mentioned above, albeit with additional technical complications.

\section{Notation and Preliminaries}\label{se:notation}

In this section, we will introduce further notation and definitions, state a useful correlation inequality, and make preliminary observations.

For the most part, our notation and terminology follow that of~\cite{BBM3D}.  In order to reduce clutter, we will omit floor signs throughout the paper.  All logarithms are taken with base~$e$.

We say that a set~$S$ is \emph{occupied\/} if it contains at least~one infected site, and \emph{empty\/} or \emph{unoccupied\/} otherwise.  If all of the sites in $S$ are infected, we say that $S$ is \emph{full}.

We will denote the vector~$(1, \ldots, 1) \in \R^{\ell}$ by $\onetotheell$.  For each $j \in [\ell]$, we let $e_j$ denote the $j$th standard basis vector.  

Given a set~$S$, we write $A \sim \Bin(S, p)$ to denote that the elements of~$A$ are chosen from $S$ independently with probability~$p$.

Harris's Lemma~\cite{Harris} will play an important role in the proof.  We define a partial order~$\leq$ on $\{0, 1\}^n$ by writing $x \leq y$ if, for all~$i \in [n]$, $x_i \leq y_i$.  We say that an event~$E \subseteq \{0,1\}^n$ is \emph{increasing\/} if, for $x$,~$y \in \{0,1\}^n$, $x \in E$ and $x \leq y$ imply that $y \in E$.  Given $p \in [0,1]$, let $\Prob_p$ denote the product measure on $\{0,1\}^n$ with $\Prob_p(i = 1) = p$ for all~$i \in [n]$.  (We will almost always suppress the dependence on $p$ and simply write $\Prob(\cdot)$.)

\begin{Harris}
If $E$ and $F$ are increasing events in $\{0,1\}^n$ and $p \in [0,1]$, then
\[
\Prob_p(E \cap F) \geq \Prob_p(E)\Prob_p(F).
\]
\end{Harris}


We conclude this section by discussing properties of the functions $\beta_k$ and~$g_k$ defined in \eqref{eq:betadef} and~\eqref{eq:gdef}, respectively.  Given $p \in (0, 1)$, we let
\begin{equation}\label{eq:qDef}
q = -\log(1 - p).
\end{equation}
Note that for $p$~sufficiently small, we have 
\[
p \leq q \leq p + p^2 \leq 2p.
\]
Equation~\eqref{eq:qDef} allows us to write
\begin{equation}\label{eq:betagRel}
\beta_k\bigl(1 - (1-p)^n\bigr) = e^{-g_k(nq)}.
\end{equation}

We also observe that~\eqref{eq:betadef} implies that
\begin{equation}\label{eq:BetaRec}
\beta_k(u)^2 = \bigl(1 - (1 - u)^k\bigr)\beta_k(u) + u(1 - u)^k.
\end{equation}

Straightforward calculations show that for all~$k$, $\beta_k$ is positive, continuous, increasing, and differentiable on $(0, 1)$ and $g_k$ is positive, continuous, decreasing, and differentiable on $(0, \infty)$.

We will need a further result about the behavior of~$g_k$.

\begin{proposition}\label{prop:gprimeBound}
For all $k \geq 1$ and all $z \geq 1$,
\begin{equation}\label{eq:gprimeBound}
\bigl\lvert g'_k(z)\bigr\rvert \leq \frac{1}{2}.
\end{equation}
\end{proposition}

The proof of Proposition~\ref{prop:gprimeBound} is given in the Appendix.


%
%
%

\section{Percolation in an Auxiliary Bootstrap Structure}\label{se:auxiliary}

In this section, we will state the key auxiliary result, Theorem~\ref{thm:cstar}, that we will use to prove Theorem~\ref{thm:UB}.  We will also define the important notion of an $L$-gap and prove a lower bound on the probability that no $L$-gaps occur in a sequence of independent events.

In Section~\ref{subsec:moreD}, we related the probability of $r$-neighbor percolation in $[n]^d$ to the probability of percolation in a set of the form $[\log_{(r-2)}(n)]^{d - r + 2} \times [2]^{r - 2}$ in which not all vertices have the same infection threshold.  So, in order to prove Theorem~\ref{thm:UB}, we will consider an alternative ``bootstrap structure'' of this form.  A \emph{bootstrap structure} is an ordered pair~$\bigl(G, (r(v))_{v \in V(G)}\bigr)$, where $G$ is a graph and $r : V(G) \to \N$.  Given a vertex~$v$, the value~$r(v)$ is called the \emph{threshold} of~$v$.  This means that if we  consider bootstrap percolation in $\bigl(G, (r(v))_{v \in V(G)}\bigr)$ and let $A_0 = A$ as before, then we have
\[
A_{t+1} = A_t \cup \bigl\{v \, : \, \lvert N(v) \cap A_t \rvert \geq r(v)\bigr\}
\]
for each $t \geq 0$.

Let $B([n]^d, r)$ denote the usual $r$-neighbor bootstrap structure on $[n]^d$.  The auxiliary bootstrap structure that we will use was defined in~\cite{BBM3D}.  Let $\Cstar{n}{d}{r}$ be the subgraph of $\mathbb{Z}^{d+\ell}$ induced by $[n]^d \times [2]^{\ell}$ in which all vertices of the form $(a_1, \ldots, a_d) \times \onetotheell$ have threshold~$r$ and all other vertices have threshold~$r + \ell$.  Note that when $\ell = 0$, this structure is the same as $B([n]^d, r)$.

Recall that $A$ denotes the set of initially infected vertices and that $[A]$ denotes the closure of~$A$, the set of vertices that ultimately become infected.  We say that $A$ \emph{semi-percolates\/} in $\Cstar{n}{d}{r}$ if $[A] \supseteq [n]^d \times \onetotheell$.  We say that a set~$S$ is \emph{internally semi-spanned\/} if $[S \cap A] \supseteq S \cap ([n]^d \times \onetotheell)$.

In order to prove Theorem~\ref{thm:UB}, we will prove a result about the probability of semi-percolation in $\Cstar{n}{d}{r}$.  Before we can state it, we need additional notation.

Letting $A \sim \Bin([n]^d \times [2]^{\ell}, p)$, we set
\begin{equation}\label{eq:Pndlrp}
P(n,d,\ell,r,p) := \Prob\bigl(A \text{ semi-percolates in } \dCstar{n}{d}{r}\bigr).
\end{equation}
(The quantity~$P(n,d,\ell,r,p)$ was originally defined in~\cite{BBM3D}.  The definition given here is slightly simpler.)

Next, we define several important constants.
For all $d \geq 2$ and $\ell \geq 0$, let
\begin{equation}\label{eq:zetaDef}
\zeta(d, \ell) = e^{-(\ell + 2)2^{2d-1}}\bigl(1 - e^{-1}\bigr)^{2d}
\end{equation}
and let
\begin{equation}\label{eq:gammaDef}
\gamma(d, \ell) = \zeta(d, \ell) e^{-d(d - 1)2^{2d-4}}.
\end{equation}
Observe that for all $d \geq 2$ and $\ell \geq 0$,
\begin{equation}\label{eq:gammaUpperBound}
\gamma(d, \ell) \leq \gamma(2, 0) = e^{-18} \bigl(1 - e^{-1}\bigr)^{4} < 10^{-8}.
\end{equation}
Finally, given $d \geq r \geq 2$ and $\ell \geq 0$, let
\begin{equation}\label{eq:constantsDef}
\constant{d}{\ell}{r} = \begin{cases}
\gamma(d, \ell), & r = 2,\\
\gamma(d - r + 2, \ell + r - 2)\Bigl(1 - \sum_{s=0}^{r-3} 2^{-r + s + 1}\Bigr), & r \geq 3.
\end{cases}
\end{equation}
We observe for later use that~\eqref{eq:constantsDef} implies that
\begin{equation}\label{eq:constantsRecursion}
\constant{d}{\ell}{r} = \constant{d - 1}{\ell + 1}{r - 1} - 2^{-r + 1}\gamma(d - r + 2, \ell + r - 2).
\end{equation}

We are at last ready to state our auxiliary result about semi-percolation in $\Cstar{n}{d}{r}$.

\begin{theorem}\label{thm:cstar}
Let $d \geq r \geq 2$, let $\ell \geq 0$,
and let $\constant{d}{\ell}{r}$ be as in~\eqref{eq:constantsDef}.  If
\begin{equation}\label{eq:pDef}
p \geq \Biggl(\dfrac{\lambda(d + \ell, \ell + r)}{\log_{(r-1)}(n)} - \dfrac{\constant{d}{\ell}{r}}{\bigl(\log_{(r-1)}(n)\bigr)^{3/2}}\Biggr)^{d-r+1},
\end{equation}
then
\[
P(n,d,\ell,r,p) \to 1
\]
as $n \to \infty$.
\end{theorem}

In Section~\ref{se:proofs}, we will show that if $p$ satisfies~\eqref{eq:pDef} for $d$, $\ell$, $r$, and $n$, then it also does so for $d - 1$, $\ell + 1$, $r - 1$, and (roughly) $\log n$.  So, by induction, if the bound on $p$ in~\eqref{eq:pDef} is sufficient for semi-percolation in $C^{\ast}([\log_{(r-2)}(n)]^{d-r+2} \times [2]^{r-2}, 2)$, then it is also sufficient for percolation in $B([n]^d, r)$.  Observe also that in order to prove Theorem~\ref{thm:UB}, it is enough to apply Theorem~\ref{thm:cstar} in the case when $\ell = 0$ (cf.~\eqref{eq:UB}).

Now let us define the notion of an $L$-gap.  For $m \geq -1$ and $\ell \geq 0$, let $\seq = (U_i)_{i \in [m + 1]} \cup (V_i^{(j)})_{i \in [m], j \in [\ell]}$ be a sequence of events.  An \emph{$L$-gap} in $\seq$\/ is an event of the form
\[
\neg \bigl(U_i \vee U_{i+1} \vee V_{i}^{(1)} \vee \ldots \vee V_{i}^{(\ell)}\bigr)
\]
for some $i \in [m]$.  (As mentioned in Section~\ref{subsec:moreD}, $L$ is not a variable, but rather refers to the shape of an $L$-gap when $d = 2$ and $\ell = 1$.)  In this paper, the events in the sequence~$\seq$ will all be of the form ``a certain set of sites is occupied''.  Thus, an $L$-gap in $\seq$ will mean that a certain collection of sets are all unoccupied.  In particular, as was the case in Section~\ref{subsec:moreD}, an $L$-gap will block the set of infected sites from growing in a specific direction.

%
%

We will need a lower bound on the probability that no $L$-gaps occur in a sequence of independent events.  We can express this bound in terms of the function~$\beta_{k}$ defined in~\eqref{eq:betadef}.  (Similar statements were proved in \cite[Lemma~6]{BBM3D} and~\cite[Proposition~10]{GravHol}.)

\begin{lemma}\label{le:LgapBound}
Let $m \geq -1$ and $\ell \geq 0$ be integers and let $u_1$, \dots,~$u_{m+1} \in (0, 1)$.  Let
\[
\seq_{m+1} := (U_i)_{i \in [m + 1]} \cup (V_i^{(j)})_{i \in [m], j \in [\ell]}
\]
be a sequence of independent events such that for each $i$, the events~$U_i$, $V_i^{(1)}$, \ldots,~$V_i^{(\ell)}$ each occur with probability~$u_i$.  Let $\mathbf{u} = (u_i)_{i=1}^{m+1}$ and let $L_{\ell}(m, \mathbf{u})$ denote the probability that no $L$-gap occurs in $\seq_{m+1}$.  If the sequence~$(u_i)_{i=1}^{m+1}$ is increasing in $i$, then
\[
L_{\ell}(m, \mathbf{u}) \geq \prod_{i=1}^{m+1} \beta_{\ell + 1}(u_i).
\]
\end{lemma}

In order to prove Lemma~\ref{le:LgapBound}, we need another result about~$\beta_{k}$.

\begin{lemma}\label{le:BetaIneq}
If $0 \leq u \leq v \leq 1$, then
\[
\bigl(1 - (1 - u)^{k}\bigr) \beta_{k}(v) + (1 - u)^{k} v \geq \beta_{k}(u) \beta_{k}(v).
\]
\end{lemma}

\begin{proof}
For $0 \leq u \leq v \leq 1$, define
\[
h(u, v) = \bigl(1 - (1 - u)^{k}\bigr) \beta_{k}(v) + (1 - u)^{k} v - \beta_{k}(u) \beta_{k}(v).
\]
and observe that we are done if we can show that $h(u, v) \geq 0$ for $u \leq v$.  By~\eqref{eq:BetaRec},
\[
\beta_k(u) h(u, v) = (1 - u)^k \bigl(v\beta_k(u) - u\beta_k(v)\bigr).
\]
Equivalently,
\[
\dfrac{\beta_k(u)}{u v} h(u, v) = (1 - u)^k \biggl(\dfrac{\beta_k(u)}{u} - \dfrac{\beta_k(v)}{v}\biggr),
\]
so it is enough to show that $\beta_k(u)/u$ is decreasing on $(0, 1)$.  Let $B = (1 - (1 - u)^k)/u$ and let $C = (1 - u)^k / u$.  By~\eqref{eq:BetaRec}, $\beta_k(u)/u$ is the positive root of $X^2 - BX - C = 0$.  Observe that both $B$ and $C$ are decreasing in $u$.  It follows that
\[
\dfrac{\beta_k(u)}{u} = \dfrac{B + \sqrt{B^2 + 4C}}{2}
\]
is also decreasing, as claimed.
\end{proof}

\begin{proof}[Proof of Lemma~\ref{le:LgapBound}.]
For $m \in \{-1, 0\}$, an $L$-gap is undefined for $\seq_{m+1}$, so, for all~$u \in (0, 1)$, we may take $L_{\ell}(-1, \mathbf{u}) = L_{\ell}(0, \mathbf{u}) = 1$.  Then, because $\beta_{\ell+1}: (0,1) \to (0,1)$, the result holds for $m \in \{-1, 0\}$.

Let $m \geq 0$ and suppose that the result holds for values smaller than~$m + 1$.
Observe that $\seq_{m+1}$ has no $L$-gaps if (i)~at least~one of the events~$U_{1}$, $V_{1}^{(1)}$, \ldots,~$V_{1}^{(\ell)}$ occurs and $(U_i)_{2 \leq i \leq m + 1} \cup (V_i^{(j)})_{2 \leq i \leq m, j \in [\ell]}$ has no $L$-gaps, or (ii)~none of these events occur, but the event~$U_2$ occurs and $(U_i)_{3 \leq i \leq m + 1} \cup (V_i^{(j)})_{3 \leq i \leq m, j \in [\ell]}$ has no $L$-gaps.  Hence, by induction,
\begin{equation}\label{eq:LgapsInductive}
L_{\ell}(m, \mathbf{u}) \geq \bigl(1 - (1 - u_{1})^{\ell + 1}\bigr) \prod_{i=2}^{m+1} \beta_{\ell + 1}(u_i) + (1 - u_{1})^{\ell + 1} u_2 \prod_{i=3}^{m+1} \beta_{\ell + 1}(u_i).
\end{equation}
Because $u_1 \leq u_2$, Lemma~\ref{le:BetaIneq} implies that
\[
\bigl(1 - (1 - u_{1})^{\ell + 1}\bigr) \beta_{\ell + 1}(u_2) + (1 - u_{1})^{\ell + 1} u_2 \geq \beta_{\ell + 1}(u_{1}) \beta_{\ell + 1}(u_{2}).
\]
Combining this with the right-hand side of~\eqref{eq:LgapsInductive} yields the claimed inequality.
\end{proof}

\section{Proof of Theorem~\ref{thm:cstar} for \texorpdfstring{$r = 2$}{r = 2}}\label{se:proofbasecase}

Our aim in this section is to prove a result that implies Theorem~\ref{thm:cstar} in the case~$r = 2$.

\begin{lemma}\label{le:baseCase}
Let $d \geq 2$, let $\ell \geq 0$, and let $\gamma(d, \ell)$ be as in~\eqref{eq:gammaDef}.  If $c$ is a constant such that
\begin{equation}\label{eq:cprimedef}
0 < c < \dfrac{3}{2}\gamma(d, \ell)
\end{equation}
and
\begin{equation}\label{eq:basecasep}
p \geq \biggl(\dfrac{\lambda(d + \ell, \ell + 2)}{\log n} - \dfrac{c}{(\log n)^{3/2}}\biggr)^{d-1},
\end{equation}
then
\[
P(n, d, \ell, 2, p) \to 1
\]
as $n \to \infty$.
\end{lemma}

\begin{remark}\label{eq:ConstantInequality}
By~\eqref{eq:constantsDef}, $\constant{d}{\ell}{2}$ certainly satisfies~\eqref{eq:cprimedef} for all $d \geq 2$ and $\ell \geq 0$.
\end{remark}

Here is a sketch of the proof of Lemma~\ref{le:baseCase}.  Below, we will define an event~$\DD_a^b$ that implies that if $[a]^d \times \onetotheell$ is internally spanned, then the infected set grows to fill $[b-1]^d \times \onetotheell$ ``diagonally'', i.e., by iteratively filling sets of the form~$[i]^d \times \onetotheell$.  The main step in the proof of the upper bound on $p_c([n]^d, 2)$ given in~\cite{BBM3D} amounts to a lower bound on $\Prob(\DD_a^b)$.  In order to prove a stronger bound on $p_c([n]^d, 2)$, we will define an event~$\TT_a^{\bvec}$ (the vector superscript is explained below) that implies that the infected set grows from $[a]^d \times \onetotheell$ to $[b]^d \times \onetotheell$ not diagonally but via a ``detour''.

We will show that $\TT_a^{\bvec}$ is not too much less probable than $\DD_a^b$ (Lemma~\ref{le:devCost}).  As the infected set grows, it may make a detour and then resume diagonal growth several times.  So, we think of the growth of the infected set as diagonal growth interrupted by a sequence of detours.  We will show that different ``growth sequences'' of this sort are disjoint events (Lemma~\ref{le:Gprops})
and that the number of growth sequences is fairly large (Lemma~\ref{le:enoughChoices}).
%
Furthermore, we will use these results to show that if $p$ satisfies~\eqref{eq:basecasep}, then with high probability some cube of the form~$[B]^d \times [2]^{\ell}$, where $B = B(p)$ is sufficiently large, is internally semi-spanned.  Finally, we will show that with high probability, such a fully infected cube leads to semi-percolation in $\Cstar{n}{d}{2}$, which will complete the proof of Lemma~\ref{le:baseCase}.

Recall the definition of an $L$-gap from Section~\ref{se:auxiliary} and recall that for $j \in [\ell]$, $e_j$ denotes the $j$th standard basis vector of~$\R^{\ell}$.  For all $i$,~$s \in \N$ and $t \in [d]$, let
\begin{equation}\label{eq:UiDef}
	U_i(t, s) = \bigl\{[s]^{t-1} \times \{i\} \times [s]^{d-t} \times \onetotheell \text{ is occupied}\bigr\},
\end{equation}
and for all $j \in [\ell]$, let
\begin{equation}\label{eq:VijDef}
	V_i^{(j)}(t, s) = \bigl\{[s]^{t-1} \times \{i\} \times [s]^{d-t} \times (\onetotheell + e_j) \text{ is occupied}\bigr\}.
\end{equation}
Let $\DD_a^b$ be the event that for each~$t \in [d]$, the sequence
\begin{equation}\label{eq:DabDef}
\bigl(U_i(t, i - 1)\bigr)_{a+1 \leq i \leq b} \cup \bigl(V_i^{(j)}(t, i - 1)\bigr)_{a+1 \leq i \leq b-1,\,j \in [\ell]}
\end{equation}
has no $L$-gaps.
%

The next result shows that, as mentioned above, the event $\DD_a^b$ means that the infected set grows ``diagonally'' from $[a]^d \times \onetotheell$ to $[b - 1]^d \times \onetotheell$.  Recall that we say that a set~$S$ is internally semi-spanned if $[S \cap A] \supseteq S \cap ([n]^d \times \onetotheell)$.

\begin{lemma}\label{le:DabGrowth}
If $[a]^d \times [2]^{\ell}$ is internally semi-spanned and $\DD_a^b$ occurs, then $[b - 1]^d \times [2]^{\ell}$ is internally semi-spanned.
\end{lemma}

\begin{proof}
We will show that if $[a]^d \times \onetotheell$ is internally spanned and $\DD_a^b$ occurs, then for each $i$ with $a + 1 \leq i \leq b - 1$, the set~$[i]^d \times \onetotheell$ is internally spanned.  We assume inductively that $[i - 1]^d \times \onetotheell$ is internally spanned.  By hypothesis, for each $t \in [d]$, the sequence in~\eqref{eq:DabDef} does not have an $L$-gap at $i$, which means that for each~$t$, all of the sites in $[i-1]^{t-1} \times \{i\} \times [i-1]^{d-t} \times \onetotheell$ become infected.  (Note that each such site already has one infected neighbor in $[i - 1]^d \times \onetotheell$.) Therefore, all of $[i]^d \times \onetotheell$ becomes infected.  This completes the proof.
\end{proof}

Let
\begin{equation}\label{eq:GabDef}
G_a^b =  \exp\Biggl[-\sum_{i=a}^{b-1} g_{\ell + 1}(i^{d-1} q) \Biggr],
\end{equation}
where $q$ is as defined in~\eqref{eq:qDef}.  Observe that if $a < b < c$, then
\begin{equation}\label{eq:GabAdditive}
G_a^c = G_a^b G_b^c.
\end{equation}

\begin{lemma}\label{cor:diagonalBound}
For all~$d \geq 2$ and all~$b > a \geq 2$,
\[
\Prob\bigl(\DD_a^b\bigr) \geq \bigl(G_a^b\bigr)^d. 
\]
\end{lemma}

\begin{proof}
Observe from \eqref{eq:UiDef} and~\eqref{eq:VijDef} that, for each $i$ and~$t$, the events $U_i(t, i-1)$, $U_{i+1}(t, i)$, $V_i^{(1)}(t, i - 1)$, \dots,~$V_i^{(\ell)}(t, i - 1)$ concern pairwise disjoint sets of sites and are therefore independent.  Furthermore, the probability that these events occur is increasing in $i$.  Hence, the sequence~\eqref{eq:DabDef} satisfies the hypotheses of Lemma~\ref{le:LgapBound}.  This and \eqref{eq:betagRel} imply that
\begin{align*}
\Prob(\DD_a^b) &\geq \Biggl(\prod_{i = a + 1}^b \beta_{\ell + 1}\Bigl(1 - (1 - p)^{(i-1)^{d-1}}\Bigr) \Biggr)^d\\
&= \exp\Biggl[-d\sum_{i = a + 1}^{b} g_{\ell + 1}\bigl((i - 1)^{d - 1}q\bigr)\Biggr]\\
&= \bigl(G_a^b\bigr)^d,
\end{align*}
as claimed.
\end{proof}

Now we will define the ``detour'' mentioned above.  In~\cite{GravHol}, Gravner and Holroyd defined an event~$\TT_a^b$ that describes another way for the infected set to grow from $[a]^2$ to~$[b]^2$. (A simplified version of~$\TT_a^b$ is shown in Figure~\ref{fig:alternative}.)  If the rows $[b - 1] \times \{a + 2\}$ and~$[b - 1] \times \{a + 3\}$ are empty, then the infected set is prevented from growing vertically.  However, if there are no double gaps in the columns to the right of~$[a]^2$, then the infected set grows horizontally until it fills the rectangle~$[b] \times [a+1]$.  If the site~$(b, a + 3)$ is infected, then the infected set overcomes the double gap and resumes vertical growth, ultimately filling $[b]^2$.

We will define a similar event~$\TT_a^{\bvec}$, where $\bvec := \{b_1, \ldots,~b_{d-1}\}$.  In this event, an $L$-gap prevents the fully infected cube~$[a+1]^d \times \onetotheell$ from growing in direction~$d$ (parts (ii), (iii), and~(v) of the definition below).  However, the infected set continues to grow in the other $d - 1$ directions (parts (vi) and~(vii)) until it meets the infected site~$\{b_{1}, \ldots,~b_{d-1}, a+3\} \times \onetotheell$ (part~(iv)).  This site allows the infected set to overcome the $L$-gap.  Finally, letting $b = \max\{b_{i} : i \in [d-1]\}$, the infected set continues to grow in direction~$d$ until it fills a cube of side length~$b$ (parts (viii) and~(ix)).

Recall the definitions of the events $U_i(t, s)$ and~$V_i^{(j)}(t, s)$ from \eqref{eq:UiDef} and \eqref{eq:VijDef}, respectively. Let $a$, $b_{1}$, \dots,~$b_{d-1}$ be such that $b := \max\{b_{i} : i \in [d-1]\}$ satisfies $b \geq a + 4$ and let $\bvec = \{b_{1}, \ldots, b_{d-1}\}$.  Define $\TT_a^{\bvec}$ to be the event that all of the following hold (see Figure~\ref{fig:tab}, which depicts the case $d = 2$, $\ell = 1$).
\begin{enumerate}[(i)]
\item For all~$t \in [d]$, the cuboid~$[a - 1]^{t-1} \times \{a + 1\} \times [a - 1]^{d-t} \times \onetotheell$ is occupied.
\item The cuboid~$[b]^{d-1} \times \{a+2\} \times \onetotheell$ is empty.
\item For all~$j \in [\ell]$, the cuboid~$[b]^{d-1} \times \{a+2\} \times (\onetotheell + e_j)$ is empty.
\item The site~$\{b_{1}, \ldots, b_{d-1}, a+3\} \times \onetotheell$ is infected.
\item The cuboid~$[b]^{d-1} \times \{a+3\} \times \onetotheell$ contains no other infected sites.
\item For all $t \neq d$, the sequence
\[
\bigl(U_i(t, a+1)\bigr)_{a+2 \leq i \leq b-1} \cup \bigl(V_i^{(j)}(t, a+1)\bigr)_{a+2 \leq i \leq b-2,\,j \in [\ell]}
\]
has no $L$-gaps.
\item For all~$t \neq d$, the cuboid~$[a+1]^{t-1} \times \{b\} \times [a+1]^{d-t} \times \onetotheell$ is occupied.
\item The sequence
\[
\bigl(U_i(d, b)\bigr)_{a+4 \leq i \leq b-1} \cup \bigl(V_i^{(j)}(d, b)\bigr)_{a+4 \leq i \leq b-2,\,j \in [\ell]}
\]
has no $L$-gaps.
\item The cuboid~$[b]^{d-1} \times \{b\} \times \onetotheell$ is occupied.
\end{enumerate}

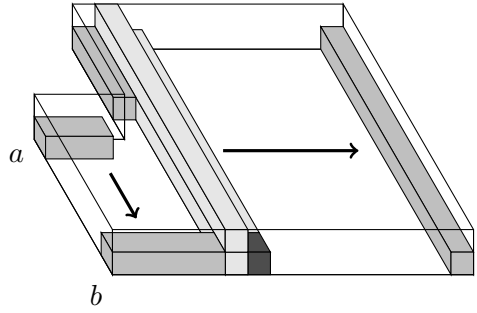
\begin{figure}
\begin{tikzpicture}[scale=0.6]
\path[shape=coordinate]
(2.5,0) coordinate(gapstart)
(2.5,0.5) coordinate(mid gapstart)
(2.5,1) coordinate(above gapstart)
(2,0) coordinate(backcorner)
(8,0) coordinate(rightcorner);
\path[xslant=tan(-30)] (backcorner) -- ++(0,-2) -- ++(-2,0) coordinate(leftcorner);
\path [name path=back]  (0,0) -- (6,0);
\path [name path=gap back] (mid gapstart) -- ++(1,0);
\path [name path=gap top,xslant=tan(-30)] (above gapstart) ++ (0.5,0) -- ++(0,-5);
\path [name path=left edge,xslant=tan(-30)] (leftcorner) -- ++(0,-3);
\path [name path=inner back,xslant=tan(-30)] (-0.5,-2) -- (2,-2);
\path [name path=bottom] (0,-5) -- (10,-5);
\path [xslant=tan(-30)]  (leftcorner) -- ++(0,-3) coordinate (bottomleft);
\path (leftcorner) -- ++(0,0.5) coordinate (mid left corner);
\path [name path=top front] (bottomleft) ++(0,1) -- ++(8,0);
\path [xslant=tan(-30)]  (gapstart) -- ++(0,-5) coordinate (gapend);
\path [xslant=tan(-30)]  (rightcorner) -- ++(0,-5) coordinate (bottomright);
\path (bottomright) -- ++(-0.5,0) coordinate (left of bottom right);
\path (rightcorner) -- ++(-0.5,0.5) coordinate (near right corner);
\path (gapend) -- ++(0,0.5) coordinate (mid gapend);
\path [name path=gap right edge,xslant=tan(-30)] (mid gapend) ++(1,0) -- ++(0,5);
\pgfmathparse{sqrt(3)/4}
\path [xslant=tan(-30)] (mid gapend) ++(1,\pgfmathresult) coordinate (infected top right corner);
\path (mid gapend) -- ++(0.5,0.5) coordinate (infected top left corner);
\node at (leftcorner) [below left] {$a$};
\node at (bottomleft) [below left] {$b$};

\draw[xslant=tan(-30)] (backcorner) -- ++(0,-2) -- ++(-2,0) -- ++(0,-3) -- ++(8,0) -- ++(0,5) --cycle;
\draw[xslant=tan(-30)] (gapstart) -- ++(0,-5);

\draw (backcorner) rectangle ++(6,1);
\pgfmathparse{5/sqrt(3)}
\draw [yslant=tan(-60)] (rightcorner) rectangle ++(\pgfmathresult,1);

\draw[fill=black!70] (gapend) ++(0.5,0.5) -- ++(0.5,0) -- (infected top right corner) -- (infected top left corner) --cycle;
\draw[fill=black!70] (gapend) ++(0.5,0) rectangle ++(0.5,0.5);

\pgfmathparse{5/sqrt(3)}
\draw[fill=gray!20,xslant=tan(-30)] (infected top right corner) rectangle ++(-0.5,4.5);
\draw[fill=gray!20,yslant=tan(-60)] (gapstart) rectangle ++(\pgfmathresult,0.5);
\draw[fill=gray!20,yslant=tan(-60)] (mid gapstart) rectangle ++(\pgfmathresult,0.5);
\draw (gapend) rectangle +(0.5,1);
\draw[fill=gray!20,xslant=tan(-30)] (above gapstart) rectangle ++ (0.5,-5);
\draw[fill=gray!20] (gapend) rectangle ++(0.5,0.5);
\draw[fill=gray!20] (mid gapend) rectangle ++(0.5,0.5);

\pgfmathparse{2/sqrt(3) - 0.25}
\draw[fill=gray!50,yslant=tan(-60)] (backcorner) rectangle ++(\pgfmathresult,0.5);
\draw[fill=gray!50] (backcorner) ++(-60:2*\pgfmathresult) rectangle ++(0.5,0.5);
\pgfmathparse{2 - sqrt(3)/4}
\draw[xslant=-tan(30),fill=gray!50] (mid gapstart) rectangle ++(-0.5,-\pgfmathresult);

\draw (leftcorner) rectangle ++(2,1);
\pgfmathparse{2/sqrt(3)}
\draw[yslant=tan(-60)] (backcorner) rectangle ++(\pgfmathresult,1);

\pgfmathparse{sqrt(3)/4}
\draw[xslant=tan(-30),fill=gray!50] (mid left corner) rectangle ++ (1.5,-\pgfmathresult);
\draw[yslant=tan(-60),fill=gray!50] (leftcorner) rectangle ++(0.25,0.5);
\draw[fill=gray!50] (leftcorner) ++(-60:0.5) rectangle ++(1.5,0.5);

\pgfmathparse{sqrt(3)/4}
\draw[xslant=tan(-30),fill=gray!50] (mid gapend) rectangle ++(-2.5,\pgfmathresult);
\draw[fill=gray!50] (bottomleft) rectangle ++(2.5,0.5);
\draw[yslant=tan(-60),fill=gray!50] (bottomleft) rectangle ++(-0.25,0.5);

\pgfmathparse{5/sqrt(3)}
\draw[xslant=tan(-30),fill=gray!50] (near right corner) rectangle ++(0.5,-5);
\draw[fill=gray!50,yslant=tan(-60)] (left of bottom right) rectangle ++(-\pgfmathresult,0.5);
\draw[fill=gray!50] (bottomright) rectangle ++(-0.5,0.5);

\draw[->,very thick,xslant=tan(-30)] (1.25,-2.75) -- ++(0,-1);
\draw[->,very thick,xslant=tan(-30)] (mid gapend) ++(1.25,2.25) -- ++(3,0) ;

\draw (bottomleft) rectangle ++(8,1);
\pgfmathparse{3/sqrt(3)}
\draw[yslant=tan(-60)] (leftcorner) rectangle ++(\pgfmathresult,1);
\end{tikzpicture}
\caption[The event~$\TT_a^{\bvec}$]{The event~$\TT_a^{\bvec}$ for $d = 2$ and $\ell = 1$.  The gray regions are occupied (parts (i), (vii), and~(ix) of the definition) and the light gray regions are unoccupied (parts (ii), (iii), and~(v)).  The dark gray cube is an infected site (part~(iv)).  The arrows depict the growth of the infected set across regions with no $L$-gaps (parts (vi) and~(viii)).}\label{fig:tab}
\end{figure}


\begin{lemma}\phantomsection\label{le:Tabprops}
\begin{enumerate}
\item[(i)] Events (i)--(ix) in the definition of\/ $\TT_a^{\bvec}$ are independent.
\item[(ii)] If $[a - 1]^d \times [2]^{\ell}$ is internally semi-spanned and\/ $\TT_a^{\bvec}$ occurs, then $[b]^d \times [2]^{\ell}$ is also internally semi-spanned.
\end{enumerate}
\end{lemma}

\begin{proof}
(i) This follows from the fact that events (i)--(ix) in the definition of~$\TT_a^{\bvec}$ concern pairwise disjoint sets of sites.  Indeed, all of the sites in the sets described in parts (i), (vi), and~(vii) have $d$th coordinate at most~$a + 1$.  Moreover, all of the sites in the sets described by the events $U_i(t, a + 1)$ and~$V_i^{(j)}(t, a + 1)$ have $i$th coordinate~$t$.  Similarly, all of the sites in the sets described in parts (ii), (iii), (iv), and~(v) have $d$th coordinate in $\{a + 2, a + 3\}$, and it is easy to see that these four~sets are pairwise disjoint.  Finally, all of the sites in parts (viii) and~(ix) have $d$th coordinate at least~$a + 4$, and it is again easy to see that the sets mentioned in these parts are pairwise disjoint.

(ii) If $[a - 1]^d \times [2]^{\ell}$ is internally semi-spanned, then part~(i)
implies that each set of the form~$[a - 1]^{t-1} \times \{a\} \times [a - 1]^{d-t} \times \onetotheell$ becomes fully infected.  This in turn guarantees that all of~$[a]^d \times \onetotheell$, and then all of~$[a + 1]^d \times \onetotheell$, becomes infected.  Parts (vi) and~(vii) then guarantee that $[b]^{d-1} \times [a + 1] \times [2]^{\ell}$ is internally semi-spanned.  Finally, parts (iv),~(viii), and~(ix) imply that $[b]^d \times [2]^{\ell}$ is internally semi-spanned.
\end{proof}


Now we will show that if $b = \max\{b_i :  i \in [d-1]\}$, then $\TT_a^{\bvec}$ is not too much less probable than $\DD_a^b$.  It will be convenient to compare $\Prob\bigl(\TT_a^{\bvec}\bigr)$ not to~$\Prob(\DD_a^b)$ but to~$G_{a}^{b}$.

\begin{lemma}\label{le:devCost}
Let $d \geq 2$, let $\ell \geq 0$, and let $\zeta = \zeta(d, \ell)$ be the constant defined in~\eqref{eq:zetaDef}.
If $p > 0$ is sufficiently small, if $a$, $b_1$, \dots,~$b_{d-1}$ are integers in the interval~$[p^{-1/(d-1)}+1, 4p^{-1/(d-1)}]$ such that $b := \max\{b_{i} : i \in [d-1]\}$ satisfies $b \geq a + 4$, and if $\bvec = \{b_1, \dots,~b_{d-1}\}$, then
\[
\Prob\bigl(\TT_a^{\bvec}\bigr) \geq \zeta p\exp\bigl[-pd(b - a)\bigl(b^{d-1} - a^{d-1}\bigr)\bigr] \bigl(G_a^b\bigr)^d.
\]
\end{lemma}

The key to the proof of Lemma~\ref{le:devCost} is that if $p$ is sufficiently small and $s$ is on the order of~$p^{-1/(d-1)}$, then  $(1-p)^{s^{d-1}}$ is bounded away from both~$0$ and~$1$.

\begin{proof}[Proof of Lemma~\ref{le:devCost}]
	By Lemmas \ref{le:LgapBound} and~\ref{le:Tabprops}(i) and the definition of~$\TT_a^{\bvec}$,
\begin{align*}
	\Prob\bigl(\TT_a^{\bvec}\bigr) &\geq \Bigl(1 - (1 - p)^{(a - 1)^{d-1}}\Bigr)^d (1 - p)^{b^{d-1}} (1 - p)^{\ell b^{d-1}} p (1 - p)^{b^{d-1} - 1} \nonumber\\
	&\qquad\times\beta_{\ell + 1}\Bigl(1 - (1 - p)^{(a+1)^{d-1}}\Bigr)^{(d-1)(b - a -2)} \Bigl(1 - (1 - p)^{(a+1)^{d-1}}\Bigr)^{d-1} \nonumber\\
	&\qquad\times \beta_{\ell + 1}\Bigl(1 - (1 - p)^{b^{d-1}}\Bigr)^{b - a - 4} \Bigl(1 - (1 - p)^{b^{d-1}}\Bigr).
\end{align*}
Because $b > a$, we have
\begin{align}\label{eq:Tbound}
	\Prob\bigl(\TT_a^{\bvec}\bigr) &\geq p (1 - p)^{(\ell + 2)b^{d - 1} - 1} \Bigl(1 - (1 - p)^{(a-1)^{d-1}}\Bigr)^{2d} \nonumber\\
	&\qquad\times\beta_{\ell + 1}\Bigl(1 - (1 - p)^{(a+1)^{d-1}}\Bigr)^{(d-1)(b - a)} \beta_{\ell + 1}\Bigl(1 - (1 - p)^{b^{d-1}}\Bigr)^{b - a}.
\end{align}
If $x$ is sufficiently small, then $e^{-x}  \geq 1 - x \geq e^{-2x}$.  So, because $b \leq 4p^{-1/(d - 1)}$ and $a - 1 \geq p^{-1/(d - 1)}$, if $p$ is sufficiently small, then
\begin{equation*}\label{eq:1-pTerms}
(1 - p)^{(\ell + 2)b^{d - 1} - 1}\Bigl(1 - (1 - p)^{(a-1)^{d-1}}\Bigr)^{2d} \geq e^{-(\ell + 2) 2^{2d-1}} \bigl(1 - e^{-1}\bigr)^{2d} = \zeta.
\end{equation*}
When we plug this into~\eqref{eq:Tbound} and use~\eqref{eq:betagRel}, we see that
\begin{align*}
	\Prob\bigl(\TT_a^{\bvec}\bigr) &\geq \zeta p \beta_{\ell + 1}\Bigl(1 - (1 - p)^{(a+1)^{d-1}}\Bigr)^{(d-1)(b - a)} \beta_{\ell + 1}\Bigl(1 - (1 - p)^{b^{d-1}}\Bigr)^{b - a}\\
	&= \zeta p \exp\Bigl[-(d - 1)(b - a)g_{\ell + 1}\bigl((a+1)^{d-1}q\bigr) - (b - a)g_{\ell + 1}\bigr(b^{d-1}q\bigr)\Bigr].
\end{align*}
Finally, since $g_{\ell + 1}$ is decreasing, we have
\begin{equation}\label{eq:2ndTbound}
\Prob\bigl(\TT_a^{\bvec}\bigr) \geq \zeta p \exp\bigl[-d(b - a)g_{\ell + 1}\bigl(a^{d-1}q\bigr)\bigr].
\end{equation}



Observe from~\eqref{eq:GabDef} that
\[
G_a^b = \exp\Biggl[-\sum_{i=a}^{b - 1} g_{\ell + 1}(i^{d-1} q)\Biggr] \leq \exp\bigl[-(b - a) g_{\ell + 1}\bigl(b^{d-1} q\bigr)\bigr].
\]
Thus, we may rewrite \eqref{eq:2ndTbound} as
\begin{equation}\label{eq:3rdTbound}
\Prob\bigl(\TT_a^{\bvec}\bigr) \geq \zeta p \exp\Bigl[-d(b - a)\Bigl(g_{\ell + 1}\bigl(a^{d-1}q\bigr) - g_{\ell + 1}\bigl(b^{d-1}q\bigr)\Bigr)\Bigr] \bigl(G_a^b\bigr)^d.
\end{equation}

Now 
\[
g_{\ell + 1}\bigl(a^{d-1}q\bigr) - g_{\ell + 1}\bigl(b^{d-1}q\bigr) \leq (b^{d-1}q - a^{d-1}q) \max_{x \in [a^{d-1}q, b^{d-1}q]} \vert g_{\ell + 1}'(x) \vert.
\]
Recall that $p \leq q$, which, by our assumptions on $a$~and~$b$, means that
$1 \leq a^{d-1}q < b^{d-1}q$.
So, Proposition~\ref{prop:gprimeBound} and the fact that $q \leq 2p$ for $p$~sufficiently small imply that
\[
g_{\ell + 1}\bigl(a^{d-1}q\bigr) - g_{\ell + 1}\bigl(b^{d-1}q\bigr) \leq \bigl(b^{d-1} - a^{d-1}\bigr)q \cdot \frac{1}{2} = p\bigl(b^{d-1} - a^{d-1}\bigr).
\]
Combining this with~\eqref{eq:3rdTbound} gives the desired result.
\end{proof}


If semi-percolation occurs in $\Cstar{n}{d}{2}$, then, as the infected set grows, it may encounter and overcome several $L$-gaps.  We order the $L$-gaps by the associated value of~$a$ and define $\bvec_{i}$ to be the vector associated with the $i$th $L$-gap.


Now we will define the event that the infected set grows from $[2]^d \times \onetotheell$ to~$[B]^d \times \onetotheell$ (where $B = B(p)$ is a large value to be chosen later) with periods of diagonal growth interrupted by a specified sequence of events of the form~$\TT_{a}^{\bvec}$.

For each $t \in [d]$, let $x^{(t)} = \{1\}^{t-1} \times \{2\} \times \{1\}^{d-t} \times \onetotheell$ and let $y^{(t)} = \{B\}^{t-1} \times \{1\} \times \{B\}^{d-t} \times \onetotheell$.  Let $m \in \N$ and let $2 \leq a_{1} < b_{1} \leq \ldots \leq a_{m} < b_{m}$ be such that for all~$i \in [m]$, we have $b_{i} - a_{i} \geq 4$, and such that $B > b_m$.   For each $i \in [m]$, let $\bvec_i = \{b_{i,1}, \dots, b_{i,d-1}\}$ be such that $b_{i} = \max \{b_{i, t} : t \in [d - 1]\}$.  Define
\begin{align*}
\growth(a_1, \bvec_1, \ldots, a_m, \bvec_m) = &\bigl(\{1\}^{d} \times \onetotheell \text{ is infected}\bigr) \cap \Biggl(\bigcap_{t=1}^d \bigl(x^{(t)} \text{ is infected}\bigr)\Biggr) \\
&\qquad \cap \DD_2^{a_1} \cap \TT_{a_1}^{\bvec_1} \cap \cdots \cap \DD_{b_{m-1}}^{a_{m}} \cap \TT_{a_m}^{\bvec_m} \cap \DD_{b_m}^{B-1}\\
&\qquad \cap \Biggl(\bigcap_{t=1}^d \bigl(y^{(t)} \text{ is infected}\bigr)\Biggr).
\end{align*}


\begin{lemma}\phantomsection\label{le:Gprops}
\begin{enumerate}
\item[(i)] The events in the definition of\/ $\growth((a_i, \bvec_i)_{i=1}^m)$ are independent.
\item[(ii)] If\/ $\growth((a_i, \bvec_i)_{i=1}^m)$ occurs then $\Cstar{B}{d}{2}$ is internally semi-spanned.
\item[(iii)] Events of the form\/ $\growth((a_i, \bvec_i)_{i=1}^m)$ are pairwise disjoint, that is, they correspond to pairwise disjoint subsets of~$\{0, 1\}^{B^d 2^{\ell}}$.
\end{enumerate}
\end{lemma}

\begin{proof}
(i) It follows from the definition of~$\DD_a^b$ and from Lemma~\ref{le:Tabprops}(i) that the events in the definition of~$\growth((a_i, \bvec_i)_{i=1}^m)$ involve pairwise disjoint sets of sites.  Thus, they are independent.

(ii) First, if all of the sites $\{1\}^d \times \onetotheell$, $x^{(1)}$, \ldots,~$x^{(d)}$ are infected, then $[2]^{d + \ell}$ is internally semi-spanned.  Next, observe that by Lemmas \ref{le:DabGrowth} and~\ref{le:Tabprops}(ii), if the events $\DD_2^{a_1}$, $\TT_{a_1}^{\bvec_1}$, \dots,~$\DD_{b_{m-1}}^{a_{m}}$, $\TT_{a_m}^{\bvec_m}$, and~$\DD_{b_{m}}^{B - 1}$ all occur, then $[B - 2]^d \times [2]^{\ell}$ is internally semi-spanned.  Finally, if all of the sites $y^{(1)}$, \ldots,~$y^{(d)}$ are infected, then $[B]^d \times [2]^{\ell}$ is internally semi-spanned.

(iii) Consider two sequences $(a_i, \bvec_i)_{i=1}^m$ and~$(a'_i, \bvec'_i)_{i=1}^m$ and the associated events $\growth((a_i, \bvec_i)_{i=1}^m)$ and~$\growth((a'_i, \bvec'_i)_{i=1}^m)$.
Recall the definitions of the events~$U_i(t, s)$ and~$V_i^{(j)}(t, s)$ from \eqref{eq:UiDef} and~\eqref{eq:VijDef}, respectively.
Given $i \geq 1$, it follows from the definition of~$\DD_a^{b}$ and parts (ii), (iii), and~(v) of the definition of~$\TT_a^{\bvec}$ that $a_i + 2$ is the least~$s \geq b_{i-1}$ such that the events~$U_s(d, s-1)$, $U_{s+1}(d, s)$, $V_s^{(1)}(d, s-1)$, \ldots,~$V_s^{(\ell)}(d, s-1)$ all do not occur.  (Here, we set $b_0 = 2$.) This means that if $a_i \neq a'_i$, then $\growth((a_i, \bvec_i)_{i=1}^m)$ and $\growth((a'_i, \bvec'_i)_{i=1}^m)$ are disjoint.  Similarly, parts (iv) and~(v) of the definition of~$\TT_a^{\bvec}$ imply that if $\bvec_i \neq \bvec'_i$, then $\growth((a_i, \bvec_i)_{i=1}^m)$ and $\growth((a'_i, \bvec'_i)_{i=1}^m)$ are disjoint.  Thus, the two events are disjoint unless they are identical, as claimed. 
\end{proof}


Parts (ii) and~(iii) of Lemma~\ref{le:Gprops} indicate that if we can bound from below the probability that an event of the form~$\growth((a_i, \bvec_i)_{i=1}^m)$ occurs, then a union bound will give us a lower bound on the probability of semi-percolation.  To this end, we wish to enumerate those sequences~$(a_i, \bvec_i)_{i=1}^m$ that satisfy certain conditions.  We will be interested in sequences such that
\begin{equation}\label{eq:inInterval}
 p^{-1/(d-1)}+1 \leq a_1 < b_1 \leq \ldots \leq a_m < b_m \leq 4p^{-1/(d-1)}
\end{equation}
and
\begin{equation}\label{eq:GrowthInterval}
4 \leq b_i - a_i \leq p^{-1/(2d-2)} \qquad \text{for all } i \in [m].
\end{equation}

Let us explain these conditions.  First, the lower bound on the probability of~$\TT_a^{\bvec}$ in Lemma~\ref{le:devCost} requires that $a$ and $b$ both be on the order of~$p^{-1/(d-1)}$, which corresponds to~\eqref{eq:inInterval}.

Second, we wish to show that there are many sequences~$(a_i, \bvec_i)_{i=1}^m$ such that for all~$i$, we have $b_i - a_i \leq K = K(p)$.  What, then, should $K$ and $m$ be?
Observe that~\eqref{eq:inInterval} implies that $Km \leq p^{-1/(d - 1)}$.
Moreover, we will show that, given $K$ and~$m$, the number of sequences of the desired form is roughly~$(K/mp)^m$, which is maximized when $K$ and $m$ have the same order of magnitude.  Thus, we will take both $K$ and~$m$ to be on the order of~$p^{-1/(2d - 2)}$; the former requirement is the second inequality in~\eqref{eq:GrowthInterval}.  Finally, $\TT_a^{\bvec}$ is defined only if $b \geq a + 4$.

\begin{lemma}\label{le:enoughChoices}
Let $d \geq 2$, let $\ell \geq 0$, and let  $p > 0$ be sufficiently small.  If $\gamma = \gamma(d, \ell)$ is as in~\eqref{eq:gammaDef} and
\[
m = \gamma p^{-1/(2d-2)},
\]
then the number of sequences $(a_i, \bvec_i)_{i=1}^m$ satisfying \eqref{eq:inInterval} and~\eqref{eq:GrowthInterval} is at least
\[
\biggl(\dfrac{8}{\gamma p}\biggr)^m.
\]
\end{lemma}

\begin{proof}
We construct sequences~$(a_i, \bvec_i)_{i=1}^m$ satisfying \eqref{eq:inInterval} and~\eqref{eq:GrowthInterval} as follows: we start by choosing $a_1$, \dots,~$a_m$ such that $a_1 \geq p^{-1/(d - 1)}+1$, such that $a_{i+1} \geq a_i + p^{-1/(2d-2)}$ for all $i \in [m - 1]$, and such that $a_m \leq 4p^{-1/(d-1)} - p^{-1/(2d-2)}$.  Then, for each~$i$, we choose $\bvec_i$ as follows.  First, we choose an element of $\{a_i + 4, \ldots, a_i + p^{-1/(2d - 2)}\}$ and call it $b_i$.  Then, to complete the vector~$\bvec_i$, we choose $d - 2$ elements of~$[b_i]$ with replacement.  Observe that a sequence chosen in this way indeed satisfies \eqref{eq:inInterval} and~\eqref{eq:GrowthInterval}.

Let $\mathcal{S}$ denote the set of sequences chosen above and observe that by Stirling's approximation, if $p$ is sufficiently small, then
\begin{align*}
\lvert \mathcal{S} \rvert &\geq \dbinom{3p^{-\frac{1}{d-1}} - 1 - mp^{-\frac{1}{2d - 2}}}{m} \bigl(p^{-\frac{1}{2d - 2}} - 3\bigr)^m \prod_{i=1}^m b_i^{d-2} \\
&\geq \Biggl(\dfrac{e(3 - 2\gamma) p^{-\frac{1}{d - 1}}}{m}\Biggr)^m \Bigl((1 -\gamma)p^{-\frac{1}{2d - 2}}\Bigr)^{m} \prod_{i=1}^m b_i^{d-2}.
\end{align*}
For each~$i \in [m]$, we have $b_i \geq p^{-1/(d-1)}$.  It follows from~\eqref{eq:gammaUpperBound} that
\[
\lvert \mathcal{S} \rvert \geq \Biggl(\dfrac{e(3 - 5\gamma)p^{-\frac{1}{2d-2}}}{m}\Biggr)^m \bigl(p^{-\frac{1}{d-1}}\bigr)^m \bigl(p^{-\frac{d-2}{d-1}}\bigr)^m \geq \biggl(\dfrac{8}{\gamma p}\biggr)^m,
\]
This completes the proof.
\end{proof}

\begin{remark}
Let us make two further remarks regarding Lemma~\ref{le:enoughChoices}.  First, one might also count sequences of fewer than~$m$ $L$-gaps, but it turns out that this would not significantly affect the total.  (Essentially, this is because if $M \gg m$, then $\sum_{j=1}^m (M/j)^j \leq m(M/m)^m$; for our purposes, the extra factor of~$m$ represents a negligible increase.)

Second, recall that in part~(iv) of the definition of~$\TT_a^{\bvec}$ we required that the site~$(b_{1}, \ldots, b_{d-1}, a + 3) \times \onetotheell$ be infected.  One might be tempted to define $\TT_a^{\bvec}$ so that the site~$(b, \ldots, b, a + 3) \times \onetotheell$ is infected.  However, with this alternative definition, no result similar to Lemma~\ref{le:enoughChoices} holds.  That is, there does not exist a constant~$c > 0$ such that the number of sequences~$(a_i, \bvec_i)_{i=1}^m$ satisfying \eqref{eq:inInterval} and~\eqref{eq:GrowthInterval} is at least~$(c/p)^m$---and, as the proof of Lemma~\ref{le:droplet} will show, this bound is exactly what we need.
\end{remark}

Recall the definition of~$P(n, d, \ell, r, p)$ from~\eqref{eq:Pndlrp}.  Now we will combine the results above to prove a lower bound on $P(B,d,\ell,2,p)$ for $B > 4p^{-1/(d-1)}$.  Once we have done so, we will be ready to prove Lemma~\ref{le:baseCase}.

\begin{lemma}\label{le:droplet}
Let $d \geq 2$, let $\ell \geq 0$, and let $\gamma = \gamma(d, \ell)$ be as in~\eqref{eq:gammaDef}.  If $p > 0$ is sufficiently small and $B > 4p^{-1/(d-1)}$, then
\begin{equation}\label{eq:seedISSbd}
P(B,d,\ell,2,p) \geq \exp\biggl[\dfrac{2\gamma}{p^{1/(2d-2)}} - \dfrac{d\lambda(d + \ell, \ell + 2)}{p^{1/(d-1)}}\biggr].
\end{equation}
\end{lemma}

\begin{proof}
Let $m$ be as in the statement of Lemma~\ref{le:enoughChoices} and suppose that the sequence~$(a_i, \bvec_i)_{i=1}^m$ satisfies \eqref{eq:inInterval} and~\eqref{eq:GrowthInterval}.  (By~\eqref{eq:inInterval}, $B > b_m$.)  We begin by bounding $\Prob(\growth((a_i, \bvec_i)_{i=1}^m))$ from below.  By Lemma~\ref{le:Gprops}(i), we have
\[
\Prob\bigl(\growth\bigl((a_i, \bvec_i)_{i=1}^m\bigr)\bigr) = p^{2d+1} \Prob(\DD_2^{a_1}) \Prob(\TT_{a_1}^{\bvec_1}) \cdots \Prob(\DD_{b_{m-1}}^{a_{m}})  \Prob(\TT_{a_m}^{\bvec_m}) \Prob(\DD_{b_m}^{B-1}).
\]
Recall that in Lemmas \ref{cor:diagonalBound} and~\ref{le:devCost}, we bounded $\Prob(\DD_a^b)$ and $\Prob(\TT_{a}^{\bvec})$, respectively, in terms of~$G_a^b$.  It follows from these results and~\eqref{eq:GabAdditive} that
\begin{equation}\label{eq:1stGbound}
\Prob\bigl(\growth\bigl((a_i, \bvec_i)_{i=1}^m\bigr)\bigr) \geq p^{2d+1} \bigl(G_2^{B-1}\bigr)^d \prod_{i=1}^m \Bigl(\zeta p e^{-pd(b_i - a_i)\bigl(b_i^{d-1} - a_i^{d-1}\bigr)} \Bigr). 
\end{equation}
By the Mean Value Theorem, for each~$i$, there exists $\alpha_i \in [a_i, b_i]$ such that
\[
b_i^{d-1} - a_i^{d-1} = (b_i - a_i)(d - 1)\alpha_i^{d - 2}.
\]
It then follows from \eqref{eq:inInterval} and~\eqref{eq:GrowthInterval} that
\[
(b_i - a_i)\bigl(b_i^{d-1} - a_i^{d-1}\bigr) \leq (d - 1)(b_i - a_i)^2 b_i^{d - 2} \leq (d-1)2^{d-2}p^{-1}.
\]
Plugging this into~\eqref{eq:1stGbound} and recalling the definition of~$\gamma$ from~\eqref{eq:gammaDef}
shows that
\begin{align}\label{eq:2ndGbound}
\Prob\bigl(\growth\bigl((a_i, \bvec_i)_{i=1}^m\bigr)\bigr) &\geq p^{2d+1} \bigl(G_2^{B-1}\bigr)^d \prod_{i=1}^m \Bigl(\zeta p e^{-d(d-1)2^{d-2}}\Bigr) \nonumber\\
&= p^{2d+1} \bigl(G_2^{B-1}\bigr)^d (\gamma p)^m.
\end{align}

Now let $\lambda = \lambda(d + \ell, \ell + 2)$ be as in~\eqref{eq:lambdaDef}.  Observe that \eqref{eq:GabDef}, the fact that $g_{\ell + 1}$ is decreasing, and the fact that $p \leq q$ imply that
\begin{align*}
G_2^{B-1} &= \exp\Biggl[-\sum_{i=2}^{B-2} g_{\ell + 1}\bigl(i^{d-1} q\bigr)\Biggr] \\
&\geq \exp\biggl[-\dfrac{1}{p^{1/(d-1)}} \int_0^{\infty} g_{\ell + 1}\bigl(z^{d-1}\bigr)\,dz\biggr] \\
&= \exp\biggl[-\dfrac{\lambda}{p^{1/(d-1)}}\biggr].
\end{align*}
Plugging this into~\eqref{eq:2ndGbound}, we see that
\begin{equation}\label{eq:3rdGbound}
\Prob\bigl(\growth\bigl((a_i, \bvec_i)_{i=1}^m\bigr)\bigr) \geq p^{2d+1} (\gamma p)^m \exp\biggl[-\dfrac{d\lambda}{p^{1/(d-1)}}\biggr].
\end{equation}

Now we are ready to prove our lower bound on $P(B,d,\ell,2,p)$.  Let $\mathcal{S}$ denote the set of sequences~$(a_i, \bvec_i)_{i=1}^m$ that satisfy \eqref{eq:inInterval} and~\eqref{eq:GrowthInterval} and recall from Lemma~\ref{le:enoughChoices} that
\[
\lvert \mathcal{S} \rvert \geq \biggl(\dfrac{8}{\gamma p}\biggr)^m.
\]
It then follows from Lemma \ref{le:Gprops}(ii),~(iii) and from~\eqref{eq:3rdGbound} that
\begin{equation*}
P(B,d,\ell,2,p) \geq \sum_{(a_i, \bvec_i)_{i=1}^m \in \mathcal{S}} \Prob\bigl(\growth\bigl((a_i, \bvec_i)_{i=1}^m\bigr)\bigr) \geq p^{2d+1} 2^m \exp\biggl[-\dfrac{d\lambda}{p^{1/(d-1)}}\biggr].
\end{equation*}
Recall from Lemma~\ref{le:enoughChoices} that $m = \gamma p^{-1/(2d-2)}$.  Since $\log(1/p) \ll p^{-1/(2d - 2)}$ for $p$~sufficiently small,
it follows that
\begin{align*}
P(B,d,\ell,2,p) &\geq p^{2d+1} \exp\biggl[\dfrac{\gamma \log 8}{p^{1/(2d-2)}} - \dfrac{d\lambda}{p^{1/(d-1)}}\biggr] \\
&\geq \exp\biggl[\dfrac{2\gamma}{p^{1/(2d-2)}} - \dfrac{d\lambda}{p^{1/(d-1)}}\biggr],
\end{align*}
as claimed.
\end{proof}

Now we will show that the right-hand side of~\eqref{eq:seedISSbd} is large enough
that it is very likely that some fairly large cube in $[n]^d \times [2]^{\ell}$ is internally semi-spanned.  In particular, the $2\gamma p^{-1/(2d-2)}$ term in the exponent on the right-hand side of~\eqref{eq:seedISSbd} will allow us to prove Lemma~\ref{le:baseCase}.

\begin{proof}[Proof of Lemma~\ref{le:baseCase}.]
Recall that we want to show that if $c$ satisfies~\eqref{eq:cprimedef} and $p$ satisfies~\eqref{eq:basecasep}, then $P(n, d, \ell, 2, p) \to 1$ as $n \to \infty$.  A standard coupling argument shows that $P(n, d, \ell, 2, p)$ is increasing in $p$, so it is enough to prove the lemma under the assumption that
\begin{equation}\label{eq:pUpperBound}
p \leq \biggl(\dfrac{\lambda(d + \ell, \ell + 2)^2}{d^2 \log n}\biggr)^{d - 1}.
\end{equation}

Let $B = p^{-3/(d-1)}$
and partition $[n]^d \times [2]^{\ell}$ into cubes of the form~$[B]^d \times [2]^{\ell}$.  We want to show that with high probability at least~one of these cubes is internally semi-spanned.  To do this, we use the following claim, whose proof we postpone to the Appendix.

\begin{claim}\label{cl:dropletbound}
Let $d \geq 2$, let $\ell \geq 0$,
and let $\gamma = \gamma(d, \ell)$ be as in~\eqref{eq:gammaDef}.
If $c$ satisfies~\eqref{eq:cprimedef} and $p$ satisfies \eqref{eq:basecasep} and~\eqref{eq:pUpperBound}, then there exists a constant~$\alpha > 0$ such that
\[
\dfrac{2\gamma}{p^{1/(2d-2)}} - \dfrac{d\lambda(d + \ell, \ell + 2)}{p^{1/(d-1)}} \geq \alpha(\log n)^{1/2} - d\log n
\]
for $n$~sufficiently large.
\end{claim}

Let $I_B$ denote the event that at least~one cube of the form~$[B]^d \times [2]^{\ell}$ is internally semi-spanned and let $\lambda = \lambda(d + \ell, \ell + 2)$. By Lemma~\ref{le:droplet}, Claim~\ref{cl:dropletbound}, and the fact that $e^{(\log n)^{1/3}} \gg B$, we have
\begin{align*}
\Prob(I_B) &\geq 1 - \biggl(1 - \exp\biggl[\dfrac{2\gamma}{p^{1/(2d-2)}}-\dfrac{d\lambda}{p^{1/(d-1)}}\biggr] \biggr)^{(n/B)^d}\\
		&\geq 1 - \exp \biggl[-\biggl(\dfrac{n}{B}\biggr)^d \exp\biggl(\dfrac{2\gamma}{p^{1/(2d-2)}} -\dfrac{d\lambda}{p^{1/(d-1)}} \biggr) \biggr]\\
		&\geq 1 - \exp \biggl[-\biggl(\dfrac{n}{B}\biggr)^d \exp\bigl(\alpha(\log n)^{1/2} - d\log n \bigr) \biggr]\\
		&\geq 1 - \exp \bigl[-\exp\bigl(\alpha(\log n)^{1/2} - (\log n)^{1/3}\bigr) \bigr]\\
		&= 1 - o(1).
\end{align*}

It is easy to see that if a cube of the form~$[B]^d \times [2]^{\ell}$ is internally semi-spanned and every cuboid of the form~$[B]^{t-1} \times \{1\} \times [B]^{d-t} \times \onetotheell$ is occupied, then the initially infected set~$A$ semi-percolates in $\Cstar{n}{d}{2}$.  By~\eqref{eq:pUpperBound}, for all $d \geq 2$, we have $p \ll (\log n)^{-1/2}$.  So, by the definition of~$B$, the probability that some cuboid of the form~$[B]^{t-1} \times \{1\} \times [B]^{d-t} \times \onetotheell$ is unoccupied is at most
\[
dn^d(1 - p)^{B^{d-1}} \leq dn^d e^{-p^{-2}} = o(1).
\]
Finally, because the events ``a cube of the form~$[B]^d \times [2]^{\ell}$ is internally semi-spanned'' and ``a cuboid of the form~$[B]^{t-1} \times \{1\} \times [B]^{d-t} \times \onetotheell$ is occupied'' are increasing, Harris's Lemma implies that $P(n,d,\ell,2,p) \to 1$ as $n \to \infty$.
\end{proof}

\section{Proofs of Theorems \ref{thm:UB} and~\ref{thm:cstar}}\label{se:proofs}

In this section, we complete the proof of Theorem~\ref{thm:cstar} and use it to deduce Theorem~\ref{thm:UB}.  Note that Lemma~\ref{le:baseCase} proves Theorem~\ref{thm:cstar} for $r = 2$, all $d \geq 2$, and all $\ell \geq 0$.  The rest of the proof of Theorem~\ref{thm:cstar} is an inductive argument that is very similar to the one used in~\cite{BBM3D}.  However, we face a number of technical complications not present in~\cite{BBM3D}.  So, in spite of the many similarities, we will give almost all of the details of the proof.

Recall that we wish to show that for all $d \geq r$ and all $\ell \geq 0$, if $p$ satisfies~\eqref{eq:pDef}, then $P(n, d, \ell, r, p) \to 1$ as $n \to \infty$.  We will assume that Theorem~\ref{thm:cstar} holds for all $r' < r$, all $d \geq r'$, and all $\ell \geq 0$.  In order to carry out the induction, we need two lemmas.
The first lemma is due to Holroyd~\cite[Lemma 2]{HolModAllD}.  We will need it for the case~$r = 3$.

\begin{lemma}\label{le:3Dbound}
For any $d \geq 3$, $\ell \geq 0$, and $\varepsilon > 0$, if $n$ is sufficiently large and $p^{-2d} \leq n^{\varepsilon}$, then
\[
P(n, d, \ell, 3, p) \geq \exp\bigl(-n^{1+\varepsilon}\bigr). \tag*{\qedsymbol}
\]
\end{lemma}

The second is due to Balogh, Bollob\'as, and Morris~\cite[Lemma 12]{BBM3D}.

\begin{lemma}\label{le:BBMLemma12}
For each $d \geq r \geq 2$ and each $\ell \geq 0$, there exists a constant~$\eta = \eta(d, \ell, r) > 0$ such that the following holds.  Let $\varepsilon$,~$p > 0$, let $n$,~$m \in \N$, and let $A \sim \Bin([n]^d \times [2]^{\ell},p)$.  If
\begin{equation}\label{eq:lowerISS}
P(m, d - i, \ell + i, r - i, p) \geq 1 - \eta \qquad \text{for all } i \in [r-2]
\end{equation}
and if $M \leq n$ is such that $M/m$ is sufficiently large (depending on $d$, $\ell$, $r$, and~$\varepsilon$), then
\begin{equation}\label{eq:expSmall}
\Prob\Bigl([n]^d \times \onetotheell \subseteq \bigl[A \cup \bigl([M]^d \times \onetotheell\bigr)\bigr]\Bigr) \geq 1 - \varepsilon;
\end{equation}
in particular,
\[
P(n,d,\ell,r,p) \geq (1 - \varepsilon)P(M,d,\ell,r,p). \tag*{\qedsymbol}
\]
\end{lemma}

\begin{remark}
Lemma~\ref{le:BBMLemma12} simply provides a lower bound on the the probability that the infected set grows from a smaller cuboid to a larger one.  However, the role of~$m$ in the statement of the lemma deserves some explanation.  Let $t \geq m$ and suppose that $[t]^d \times \onetotheell$ is internally spanned.  It follows from an observation in~\cite{AizLeb} that if~\eqref{eq:lowerISS} holds, then the probability that $[t+1]^d \times \onetotheell$ is not internally spanned is exponentially small in $t/m$. (For a proof of this statement, see, e.g., \cite[Lemma 11]{BBM3D}.)  So, Harris's Lemma and the assumption on $M/m$ imply that there exists~$C > 0$ such that
\[
\Prob\Bigl([n]^d \times \onetotheell \subseteq \bigl[A \cup \bigl([M]^d \times \onetotheell\bigr)\bigr]\Bigr) \geq \prod_{t=M}^{n-1} \bigl(1 - Ce^{-t/m}\bigr) \geq 1 - \varepsilon,
\]
which is exactly~\eqref{eq:expSmall}.
\end{remark}

To prove Theorem~\ref{thm:cstar}, we will need to define quantities $m$,~$M$, and~$N$ such that $1 \ll m \ll M \ll N \ll n$.  We will bound from below the probability of filling a cuboid of side length~$M$.  Then, using our induction hypothesis and Lemma~\ref{le:BBMLemma12}, we will bound the probability that this cuboid grows to fill a cuboid of side length~$N := (\log n)^3$.  Once we have bounded $P(N,d,\ell,r,p)$, it will be easy to show that, with high probability, there exists a copy of~$[N]^d \times \onetotheell$ in $[n]^d \times \onetotheell$ that is internally spanned and that, with high probability, this copy of~$[N]^d \times \onetotheell$ grows to fill all of~$[n]^d \times \onetotheell$.

In order to apply Lemma~\ref{le:BBMLemma12} in the proof of Theorem~\ref{thm:cstar}, we must take some care in choosing the values of $m$~and~$M$.
Recall that we want to define $m$ such that for all~$i \in [r - 2]$, $P(m, d - i, \ell + i, r - i, p)$ is sufficiently close to~$1$.
If $m$ is such that 
\begin{equation}\label{eq:pInductive}
p \geq \Biggl(\dfrac{\lambda(d + \ell, \ell + r)}{\log_{(r - i - 1)}(m)} - \dfrac{\constant{d - i}{\ell + i}{r - i}}{\bigl(\log_{(r - i - 1)}(m)\bigr)^{3/2}}\Biggr)^{d-r+1},
\end{equation}
then the desired lower bound on $P(m, d - i, \ell + i, r - i, p)$
follows from the induction hypothesis.  Comparing~\eqref{eq:pInductive} to the bound on $p$ in~\eqref{eq:pDef} suggests that it is reasonable to define $m$ such that $\log_{(r - 2)}(m)$ is close to~$\log_{(r - 1)}(n)$.  Recall also from Lemma~\ref{le:BBMLemma12} that we want to define $M$ such that $M/m \to \infty$ as $n \to \infty$.  Furthermore, it will turn out that we want $\log_{(r - 2)}(M)$ to be slightly less than~$\log_{(r - 1)}(n)$.  However, how close $\log_{(r - 2)}(m)$ and $\log_{(r - 2)}(M)$ must be to~$\log_{(r - 1)}(n)$ depends on~$n$, which complicates the argument slightly.

First, let
\[
N = (\log n)^3.	
\]
Given $d$,~$\ell$, and~$r$, let $\lambda = \lambda(d + \ell, \ell + r)$ be as in~\eqref{eq:lambdaDef}.
We define $M$, $m$, and a third quantity~$\delta$ such that $M$ is the largest positive value such that
\begin{align}
\delta &= \dfrac{2^{-r}\gamma(d - r + 2, \ell + r - 2)}{\lambda}\bigl(\log_{(r-2)}(M)\bigr)^{-1/2}, \label{eq:deltadef} \\
\log_{(r-2)}(M) &= (1 - \delta)\log_{(r-1)}(n), \label{eq:newMdef} \\
\intertext{and}
\log_{(r-2)}(m) &= (1 - 2\delta)\log_{(r-1)}(n). \label{eq:newmdef}
\end{align}

It is not necessarily obvious from \eqref{eq:deltadef}--\eqref{eq:newmdef} that $M$ and $m$ are well-defined.  To see that these quantities are indeed well-defined for $n$~sufficiently large, let $c = 2^{-r}\gamma(d - r + 2, \ell + r - 2) / \lambda$ and observe that by~\eqref{eq:deltadef}, we may rewrite~\eqref{eq:newMdef} as 
\begin{equation}\label{eq:Mdef}
\log_{(r-2)}(M) = \Bigl(1 - c\bigl(\log_{(r-2)}(M)\bigr)^{-1/2}\Bigr)\log_{(r-1)}(n).
\end{equation}
Now let $y = \log_{(r-1)}(n)$ and let
\[
f(x) = x - \bigl(1 - cx^{-1/2}\bigr)y.
\]
Elementary calculations show that for $n$~sufficiently large, $f$ has at least~one and at most~two positive real roots, at least~one of which is larger than~$1$.  Let $x_0$ be the larger (or only) positive real root of~$f$.  Then we may define $M$ by $\log_{(r-2)}(M) = x_0$, which is exactly~\eqref{eq:Mdef}.

%
%
%

We note that $\delta \to 0$ as $n \to \infty$.  (This convergence to~0, which we have not been able to avoid, is the source of most of the technical complications in the proof of Theorem~\ref{thm:cstar}.)  Also, observe that \eqref{eq:deltadef} and~\eqref{eq:newMdef} imply that there exists a constant~$C > 0$ such that
\begin{equation}\label{eq:deltaLB}
\delta \geq C \bigl(\log_{(r-1)} (n)\bigr)^{-1/2}
\end{equation}
for $n$~sufficiently large.

We are now ready to proceed with the proof of Theorem~\ref{thm:cstar}.

\begin{proof}[Proof of Theorem~\ref{thm:cstar}.]
As stated above, Lemma~\ref{le:baseCase} gives the result for $r = 2$.  So, suppose that $r \geq 3$ and that for all~$r' < r$, the result holds for all~$d \geq r'$ and for all~$\ell \geq 0$.

We begin by proving a lower bound on $P(M,d,\ell,r,p)$.

\begin{claim}\label{cl:Mcuboid}
We have $P(M,d,\ell,r,p) \geq 1/n$ as $n \to \infty$.
\end{claim}

\begin{proof}
To prove the claim for $r = 3$, we first observe that
\[
p^{-2d} \leq (\log\log n)^{4d^2} \leq M^{\delta}.
\]
The first inequality follows from~\eqref{eq:pDef}.  To see the second inequality, note that by~\eqref{eq:deltadef} and~\eqref{eq:deltaLB},
\[
\delta \log M \geq C'(\log\log n)^{1/2} \gg 4d^2 \log\log\log n.
\]
Then, by Lemma~\ref{le:3Dbound} and~\eqref{eq:newMdef},
\[
P(M,d,\ell,3,p) \geq \exp\bigl(-M^{1+\delta}\bigr) = \exp\Bigl(-(\log n)^{1 - \delta^2}\Bigr) \geq 1/n
\]
for $n$~sufficiently large.

To prove the claim for $r \geq 4$, it suffices to bound $P(M,d,\ell,r,p)$ from below by the probability that $[M]^d \times \onetotheell$ is full.  To do this, we first show that
\begin{equation}\label{eq:dlogM}
\log M \ll (1 - \delta)\log\log n.
\end{equation}
Observe that for all $k \geq 2$,
\begin{equation}\label{eq:loginequality}
\bigl(\log_{(k)} (n)\bigr)^{1-\delta} \ll (1-\delta)\log_{(k)} (n).
\end{equation}
(To see that~\eqref{eq:loginequality} indeed holds for $n$~sufficiently large, take logarithms and note that $\log(1 - \delta) \geq -2\delta$ for $\delta$~sufficiently small.)
If we iteratively exponentiate both sides of~\eqref{eq:newMdef} and apply~\eqref{eq:loginequality},
we see that for all $i \leq r - 3$,
\begin{equation*}
\log_{(r - 2 - i)}(M) \leq \bigl(\log_{(r - 1 - i)}(n)\bigr)^{1 - \delta} \ll (1 - \delta)\log_{(r - 1 - i)}(n).
\end{equation*}
This yields~\eqref{eq:dlogM}.

We then observe that, by~\eqref{eq:dlogM},
\[
P(M,d,\ell,r,p) \geq p^{M^d} \geq \exp\bigl(-\log(1/p)(\log n)^{1 - \delta}\bigr),
\]
which means that we are done if we can show that  
\begin{equation}\label{eq:sizeMcalc}
\exp\bigl(-\log(1/p)(\log n)^{1 - \delta}\bigr) \gg \dfrac{1}{n}.
\end{equation}
If we take logarithms twice in~\eqref{eq:sizeMcalc}, we see that it is enough to show that $\log\log(1/p) \ll \delta \log\log n$.  Observe that~\eqref{eq:deltaLB} and the fact that $1/p \leq c'\log_{(r-1)}(n)$ imply that for all~$r \geq 4$, we have
\[
\delta \log\log n \geq C' \dfrac{\log\log n}{\bigl(\log_{(r-1)}(n)\bigr)^{1/2}} \gg \log_{(r+1)}(n) \geq \log\log(1/p)
\]
as required to prove~\eqref{eq:sizeMcalc}.  This proves the claim.
\end{proof}

Now we wish to use Lemma~\ref{le:BBMLemma12} to show that $P(N,d,\ell,r,p) \geq 1/2n$ for all $n$~sufficiently large.  

First, we claim that $M/m \to \infty$ as $n \to \infty$.  For $r \geq 4$, this is easy to see.  For $r = 3$, we observe that by \eqref{eq:newMdef}, \eqref{eq:newmdef}, and~\eqref{eq:deltaLB},
\[
\dfrac{M}{m} = (\log n)^{\delta} = \exp(\delta \log\log n) \geq \exp\bigl(C'(\log\log n)^{1/2}\bigr),
\]
which tends to infinity as $n \to \infty$.

Next, we show that our induction hypothesis implies that~\eqref{eq:lowerISS} holds, i.e., that $A$ is likely to semi-percolate in the lower-threshold sets adjacent to $[m]^d \times [2]^{\ell}$.
Once we have done so, we will be ready to apply Lemma~\ref{le:BBMLemma12}.

\begin{claim}\label{cl:lowerthresh}
For all~$i \in [r - 2]$, $P(m,d - i,\ell + i,r - i,p) \to 1$ as $n \to \infty$.
\end{claim}


\begin{proof}
Let $\lambda = \lambda(d + \ell, \ell + r)$.  By induction, it is enough to show that for all~$i \in [r - 2]$, with $\constant{d}{\ell}{r}$ as in~\eqref{eq:constantsDef}, we have
\begin{align*}
p &\geq \Biggl(\dfrac{\lambda(1 - 2\delta)}{\log_{(r - 2)}(m)} - \dfrac{\constant{d}{\ell}{r}(1 - 2\delta)^{3/2}}{\bigl(\log_{(r - 2)}(m)\bigr)^{3/2}}\Biggr)^{d-r+1} \\
	&\geq \Biggl(\dfrac{\lambda}{\log_{(r - i - 1)}(m)} - \dfrac{\constant{d - i}{\ell + i}{r - i}}{\bigl(\log_{(r - i - 1)}(m)\bigr)^{3/2}}\Biggr)^{d-r+1},
\end{align*}
where the first inequality follows from \eqref{eq:pDef} and~\eqref{eq:newmdef}.  For $i \geq 2$, the second inequality is easy to see.  For $i = 1$, we need to show that 
\[
\dfrac{\lambda(1 - 2\delta)}{\log_{(r - 2)}(m)} - \dfrac{\constant{d}{\ell}{r}(1 - 2\delta)^{3/2}}{\bigl(\log_{(r - 2)}(m)\bigr)^{3/2}} \geq \dfrac{\lambda}{\log_{(r - 2)}(m)} - \dfrac{\constant{d - 1}{\ell + 1}{r - 1}}{\bigl(\log_{(r - 2)}(m)\bigr)^{3/2}}.
\]
Because $(1 - 2\delta)^{3/2} < 1$, it is enough to show that
\begin{equation*}
2\delta\lambda (\log_{(r - 2)}(m))^{1/2} + \constant{d}{\ell}{r} \leq \constant{d - 1}{\ell + 1}{r - 1}.
\end{equation*}
Indeed, \eqref{eq:deltadef}, the fact that $m \leq M$, and~\eqref{eq:constantsRecursion} imply that
\[
2\delta\lambda (\log_{(r - 2)}(m))^{1/2} + \constant{d}{\ell}{r} \leq 2^{-r+1}\gamma(d - r + 2, \ell + r - 2) + \constant{d}{\ell}{r} = \constant{d - 1}{\ell + 1}{r - 1}.
\]

It then follows from the induction hypothesis that for each $i \in [r-2]$, $P(m,d - i,\ell + i,r - i,p) \to 1$
as $n \to \infty$, as claimed.
\end{proof}

By Claim~\ref{cl:lowerthresh}, we may apply Lemma~\ref{le:BBMLemma12} to $P(N, d, \ell, r, p)$.  The lemma and  Claim~\ref{cl:Mcuboid} imply that
\[
P(N,d,\ell,r,p) \geq (1 - \varepsilon)P(M,d,\ell,r,p) \geq \dfrac{1}{2n}
\]
for all $n$~sufficiently large.  Since $(n/N)^d \gg 2n$, with high probability, there exists a cuboid $K \times \onetotheell \subseteq [A]$ with $\vert K \vert \geq N^d$.  So, by applying Lemma~\ref{le:BBMLemma12} again (this time with $N$ in place of~$M$) and Harris's Lemma, we have
\[
P(n,d,\ell,r,p) \geq \bigl(1 - o(1)\bigr) \Prob\Bigl([n ]^d \times \onetotheell \subseteq \bigl[A \cup \bigl(K \times \onetotheell\bigr)\bigr]\Bigr) = 1 - o(1).
\]
This completes the proof of Theorem~\ref{thm:cstar}.
\end{proof}

The proof of Theorem~\ref{thm:UB} is immediate.

\begin{proof}[Proof of Theorem~\ref{thm:UB}.]
Let $d \geq r \geq 2$ and let $n$ be sufficiently large.  Let $c_{d,r} = \constant{d}{0}{r}$ and note that~\eqref{eq:constantsDef}
implies that $c_{d,r} > 0$.  Applying Theorem~\ref{thm:cstar} with $\ell = 0$ shows that
\[
p_c([n]^d, r) \leq \Biggl(\dfrac{\lambda(d,r)}{\log_{(r-1)}(n)} - \dfrac{c_{d,r}}{\bigl(\log_{(r-1)}(n)\bigr)^{3/2}}\Biggr)^{d-r+1},
\]
as claimed.
\end{proof}

\section{Open Questions}\label{se:open}

It remains to improve the lower bound on the critical probability~$p_c([n]^d, r)$ for values of $d \geq r \geq 2$ other than $d = r = 2$.  Given the difficulty of the proof of the lower bound in Theorem~\ref{thm:AllDthreshold}, this is likely to be much harder than the proof of Theorem~\ref{thm:UB}, especially for $r \geq 3$.  Note, however, that the upper bound on $p_c([n]^2, 2)$ in~\cite{GravHol} gave the correct order of magnitude of the second term.  Because the proof of Theorem~\ref{thm:UB} can be seen as a fairly natural generalization of the arguments in~\cite{GravHol} to higher dimensions,
we conjecture that it gives the correct order of magnitude of the second term in $p_c([n]^d, r)$ for all~$d \geq r \geq 2$.

\begin{conjecture}\label{conj:allDconj}
Let $d \geq r \geq 2$.  As $n \to \infty$,
\[
p_c\bigl([n]^d, r\bigr) = \Biggl(\dfrac{\lambda(d,r)}{\log_{(r-1)}(n)} - \Theta\Biggl(\dfrac{1}{\bigl(\log_{(r-1)}(n)\bigr)^{3/2}}\Biggr)\Biggr)^{d-r+1}.
\]
\end{conjecture}

\section{Acknowledgments}

I am grateful to Rob Morris for pointing out that an earlier version of Claim~\ref{cl:dropletbound} could be strengthened, which made it possible to strengthen the statement of Theorem~\ref{thm:UB}.  I also wish to thank Paul Balister, Richard Johnson, and Micha\l{} Przykucki for reading preliminary versions of the manuscript and for numerous helpful comments.  Finally, I am grateful to the anonymous referees for many helpful comments that improved the presentation of the paper.

\bibliographystyle{amsplain}
\bibliography{Bootstrapbib,ProbBib}

\appendix
\section{}\label{appendix}

Here we give the proofs of results from the paper that, while straightforward, rely on somewhat lengthy calculations.

\begin{proof}[Proof of Proposition~\ref{prop:gprimeBound}.]
Recall that we wish to bound $\lvert g'_k(z) \rvert$ from above for all $k \geq 1$ and all $z \geq 1$.  By~\eqref{eq:gdef},
\begin{equation}\label{eq:gprimeAbsoluteValue}
\bigl\lvert g'_k(z) \bigr\rvert = \biggl\lvert \dfrac{e^{-z}\beta'_k\bigl(1 - e^{-z}\bigr)}{\beta_k\bigl(1 - e^{-z}\bigr)} \biggr\rvert.
\end{equation}

We begin by bounding $\beta'_k(1 - e^{-z})$ from above.  (Recall that $\beta_k$ is increasing on $(0, 1)$.)  First, differentiating both sides of~\eqref{eq:BetaRec} gives
\[
2\beta_k(u) \beta'_k(u) = k(1 - u)^{k-1} \beta_k(u) + \bigl(1 - (1 - u)^k\bigr)\beta'_k(u) + (1 - u)^k - ku(1 - u)^{k-1}.
\]
We may rewrite this as
\[
\beta'_k(u) = \dfrac{k(1 - u)^{k-1} \beta_k(u) + (1 - u)^k - ku(1 - u)^{k-1}}{2\beta_k(u) - 1 + (1 - u)^k}.
\]
It follows from~\eqref{eq:betadef} and the fact that $\beta_k(u) < 1$ for all $u \in (0, 1)$ that
\begin{align}\label{eq:betaprimeBound}
\beta'_k(u) &= \dfrac{k(1 - u)^{k-1} \beta_k(u) + (1 - u)^k - ku(1 - u)^{k-1}}{\sqrt{1 + (4u - 2)(1 - u)^k + (1 - u)^{2k}}} \nonumber\\
&\leq \dfrac{k(1 - u)^{k-1} + (1 - u)^k - ku(1 - u)^{k-1}}{\sqrt{1 + (4u - 2)(1 - u)^k + (1 - u)^{2k}}} \nonumber\\
&= \dfrac{(k + 1)(1 - u)^k}{\sqrt{1 + (4u - 2)(1 - u)^k + (1 - u)^{2k}}}.
\end{align}
Observe that the denominator of the right-hand side of~\eqref{eq:betaprimeBound} is at least~$1$ for all $u \geq 1/2$.  If $z \geq 1$, then $1 - e^{-z} \geq 1/2$, so for all such $z$, we have
\begin{equation}\label{eq:betaprimeUpperBound}
\beta'_k\bigl(1 - e^{-z}\bigr) \leq (k + 1)e^{-zk}.
\end{equation}

Next, observe that for $u \geq 0$, the quantity under the square root on the right-hand side of
~\eqref{eq:betadef} is at least~$(1 - (1 - u)^k)^2$, which means that
\begin{equation}\label{eq:betaLowerBound}
\beta_k(u) \geq 1 - (1 - u)^k
\end{equation}
for all $u \in (0, 1)$.

When we combine \eqref{eq:betaprimeUpperBound} and~\eqref{eq:betaLowerBound} with~\eqref{eq:gprimeAbsoluteValue}, we find that
\[
\bigl\lvert g'_k(z) \bigr\rvert \leq \dfrac{(k + 1)e^{-z(k + 1)}}{1 - e^{-zk}} = \dfrac{k + 1}{e^{z(k + 1)} - e^z} \leq \frac{2}{e^2 - 2} < \frac{1}{2},
\]
which is what we wanted.
\end{proof}

\begin{proof}[Proof of Claim~\ref{cl:dropletbound}.]
The claim is a lower bound on $\log(P(B, d, \ell, 2, p))$.  Recall that $c$ is the constant from Lemma~\ref{le:baseCase} and that
\begin{equation}\label{eq:pBounds}
\biggl(\dfrac{\lambda(d + \ell, \ell + 2)}{\log n} - \dfrac{c}{(\log n)^{3/2}}\biggr)^{d-1} \leq p \leq \biggl(\dfrac{\lambda(d + \ell, \ell + 2)^2}{d^2 \log n}\biggr)^{d - 1},
\end{equation}
where the upper bound is the assumption~\eqref{eq:pUpperBound}.

Let $\lambda = \lambda(d + \ell, \ell + 2)$. By~\eqref{eq:pBounds},
\begin{equation*}\label{eq:dropletBoundDisplay}
\dfrac{2\gamma}{p^{1/(2d-2)}} - \dfrac{d\lambda}{p^{1/(d-1)}} \geq 2\gamma\biggl(\dfrac{\lambda^2}{d^2 \log n}\biggr)^{-1/2} - d\log n \biggl(1 - \dfrac{c}{\lambda(\log n)^{1/2}}\biggr)^{-1}.
\end{equation*}
Let $\varepsilon > 0$ be sufficiently small.  If $x$ is sufficiently small, then $(1 - x)^{-1} \leq 1 + (1 + \varepsilon) x$. Hence, for $n$~sufficiently large, we have
\begin{equation}\label{eq:DropletBoundRewrite}
\dfrac{2\gamma}{p^{1/(2d-2)}} - \dfrac{d\lambda}{p^{1/(d-1)}} \geq d\log n \cdot \dfrac{2\gamma}{\lambda(\log n)^{1/2}} - d\log n \biggl(1 + \dfrac{(1 + \varepsilon)c}{\lambda(\log n)^{1/2}}\biggr).
\end{equation}
By~\eqref{eq:cprimedef},
\[
(1 + \varepsilon)c < \dfrac{4c}{3} < 2\gamma.
\]
It follows that there exists $\alpha > 0$ such that the right-hand side of~\eqref{eq:DropletBoundRewrite} is at least~$\alpha(\log n)^{1/2} - d\log n$, which is what we wanted.
\end{proof}

\end{document}